\documentclass[12pt]{amsart}
\usepackage{amssymb}
\usepackage{amsthm}
\usepackage[margin=1in]{geometry}
\usepackage{amstext}
\usepackage{amsbsy}
\usepackage{amscd}
\usepackage{enumerate}
\usepackage{chngpage}
\usepackage{mathtools}
\usepackage{amsmath}
\usepackage{hyperref}
\usepackage{tikz}
\usepackage{stmaryrd}
\usepackage{color}

\newtheorem{thm}{Theorem}[section]
 \newtheorem{cor}[thm]{Corollary}
\newtheorem{prop}[thm]{Proposition}
\theoremstyle{definition}
\newtheorem{defn}[thm]{Definition}
\newtheorem{question}[thm]{Question}

\makeatletter                                        
\let\@wraptoccontribs\wraptoccontribs                
\makeatother                                         
\newtheorem{thmA}{Theorem}[section]                  
\newtheorem{defnA}[thmA]{Definition A\ignorespaces}  
\newtheorem{rmkA}[thmA]{Remark A\ignorespaces}       
\theoremstyle{plain}                                 
\newtheorem{conjA}[thmA]{Conjecture A\ignorespaces}  

\newcommand{\Z}{\mathbb Z}

\newcommand{\R}{\mathbb R}

\newcommand{\HH}{\mathcal{H}}

\DeclareMathOperator{\GL}{GL}
\DeclareMathOperator{\Span}{Span}

\DeclareMathOperator{\corank}{corank}

\DeclareMathOperator{\Aut}{Aut}
\DeclareMathOperator{\col}{col}
\DeclareMathOperator{\cok}{cok}

\title{Most subrings of $\Z^n$ have large corank}
\author{Kelly Isham}
\address{Colgate University, Hamilton, NY 13346}
\email{kisham@colgate.edu}
\author{Nathan Kaplan}
\address{University of California, Irvine, CA 92697}
\email{nckaplan@math.uci.edu}

\address{City College of New York, New York, NY, 10031}
\email{gchinta@ccny.cuny.edu}

\date{\today}
\begin{document}
	
	\maketitle

\vspace{-.5cm}
{\centerline{(With an Appendix by Gautam Chinta)}}

	\begin{abstract}
If $\Lambda \subseteq \Z^n$ is a sublattice of index $m$, then $\Z^n/\Lambda$ is a finite abelian group of order $m$ and rank at most $n$.  Several authors have studied statistical properties of these groups as we range over all sublattices of index at most $X$.  In this paper we investigate quotients by sublattices that have additional algebraic structure.  While quotients $\Z^n/\Lambda$ follow the Cohen-Lenstra heuristics and are very often cyclic, we show that if $\Lambda$ is actually a subring, then once $n \ge 7$ these quotients are very rarely cyclic.  More generally, we show that once $n$ is large enough the quotient typically has very large rank. In order to prove our main theorems, we combine inputs from analytic number theory and combinatorics.  We study certain zeta functions associated to $\Z^n$ and also prove several results about matrices in Hermite normal form whose columns span a subring of $\Z^n$.
	\end{abstract}
	
	\section{Introduction}
	
	The focus of this article is about a family of random finite abelian groups that arise in number theory.  There has recently been extensive interest in this subject; see for example the survey of Wood \cite{WoodICM}.  We show that quotients of $\Z^n$ by random subrings do not look like quotients of $\Z^n$ by random sublattices.  This has an interpretation in terms of the distribution of cokernels of families of random integer matrices and the Cohen-Lenstra heuristics, a research topic that has been quite active in recent years.  
	
	In order to describe our results, we introduce some notation.  A \emph{sublattice} $\Lambda \subseteq \Z^n$ is a finite index subgroup, that is, we reserve this term for full-rank sublattices of $\Z^n$.  For vectors $u = (u_1,\ldots, u_n)$ and $w = (w_1,\ldots, w_n)$ in $\Z^n$, we write $u \circ w$ for the vector given by the componentwise product, $u \circ w = (u_1 w_1,\ldots, u_n w_n)$.  A sublattice $\Lambda \subseteq \Z^n$ is \emph{multiplicatively closed} if $u,w \in \Lambda$ implies $u \circ w \in \Lambda$.  A multiplicatively closed sublattice is a \emph{subring} if it also contains the multiplicative identity $(1,1,\ldots, 1)$.  In this article, we address questions of the following type.
\begin{question}
Let $X$ be a positive real number and $R \subseteq \Z^n$ be a subring of index at most $X$ chosen uniformly at random.  What does the quotient $\Z^n/R$ `look like'?  For example, as $X \rightarrow \infty$ how often is this group cyclic?
\end{question}
A main point of this article is to show that the additional algebraic structure possessed by subrings leads these quotients to very often have large rank.  In order to set up the contrast with quotients of random sublattices, we highlight some results about $\Z^n/\Lambda$ where $\Lambda \subseteq \Z^n$ is a sublattice of index at most $X$ chosen uniformly at random.

A finite abelian group $G$ can be written uniquely in terms of its invariant factors,
	\[
	G \cong \Z/\alpha_1 \Z \oplus \Z/\alpha_2 \Z \oplus \cdots \oplus \Z/\alpha_n\Z,
	\]
	where $\alpha_{i+1} \mid \alpha_i$ for $1 \le i \le n-1$.  The \emph{rank} of $G$ is the largest $i$ for which $\alpha_i > 1$.  The \emph{corank} of a sublattice $\Lambda \subseteq \Z^n$ is the rank of $\Z^n/\Lambda$.  A sublattice $\Lambda \subseteq \Z^n$ is \emph{cocyclic} if either $\Lambda = \Z^n$ or $\Lambda$ has corank $1$.  
	
	Nguyen and Shparlinski compute the proportion of sublattices of $\Z^n$ that are cocyclic \cite{ns}, a result which also follows by a different method from earlier work of Petrogradsky \cite{petro}.  Chinta, Kaplan, and Koplewitz generalize this result \cite{CKK}, determining the proportion of sublattices of $\Z^n$ that have corank at most $k$ for any $k \in [1,n]$.  Throughout this paper we use $p$ to denote a prime number and write $\prod_p$ for a product over all primes.
	\begin{thm}\label{CKK_thm}\cite[Corollary 1.2]{CKK}\label{thm:lattice_corank}
		Let $n,k$ be positive integers satisfying $1 \le k \le n$.  Then,
		\begin{eqnarray*}
			p_{n,k} & := & \lim_{X \rightarrow \infty} \frac{\#\left\{\text{Sublattices } \Lambda \subseteq \Z^n\colon [\Z^n\colon \Lambda] < X \text{ and } \Lambda \text{ has corank at most } k\right\}}{\#\left\{\text{Sublattices }  \Lambda \subseteq \Z^n\colon [\Z^n\colon \Lambda] < X\right\}} \\
			& = &  
			\prod_p \left( \prod_{j=1}^n (1-p^{-j})^2
			\cdot
			\sum_{i=0}^k \frac{1}{p^{i^2} \prod_{j=1}^i (1-p^{-j})^2 \prod_{j=1}^{n-i} (1-p^{-j})}\right).
		\end{eqnarray*}
	\end{thm}
	These probabilities arise in the famous Cohen--Lenstra heuristics, which were developed during the study of statistical questions about Sylow $p$-subgroups of class groups of families of number fields \cite{cohen_lenstra}.  Let $P_n$ be the distribution on finite abelian $p$-groups of rank at most $n$ that chooses a finite abelian $p$-group $G$ of rank $r \le n$ with probability
	\begin{equation}\label{eqn:PnG}
	P_n(G) = \frac{1}{\#\Aut(G)}\left(\prod_{i=1}^n (1-p^{-i})\right)\left(\prod_{i=n-r+1}^n (1-p^{-i})\right).
	\end{equation}
	Theorem \ref{CKK_thm} states that the proportion of $\Lambda \subseteq \Z^n$ that have corank at most $k$ is equal to the product over all primes of the probability that a finite abelian $p$-group chosen from the distribution $P_n$ has rank at most $k$.  These connections are explained in \cite{CKK}.  Even for small values of $k$, when $n$ goes to infinity these probabilities are quite large.  For example, for large $n$ the proportion of cocyclic sublattices $\Lambda \subseteq \Z^n$ is approximately $85\%$, the proportion with corank at most $2$ is approximately $99.4\%$, and the proportion with corank at most $3$ is approximately $99.995\%$.

	The main theorem of this paper is that for large $n$, while the vast majority of random sublattices of $\Z^n$ have small corank, it is extremely rare for a random subring of $\Z^n$ to have small corank.  
	\begin{thm}\label{proportion_corank_k}
		Let $n$ and $k$ be positive integers satisfying $1 \le k < n$.  Define
		\[
		p^R_{n,k} = \lim_{X\rightarrow \infty} \frac{\#\left\{\text{Subrings }  R \subseteq \Z^n\colon [\Z^n\colon R] < X \text{ and } R \text{ has corank at most } k\right\}}{\#\left\{\text{Subrings }  R \subseteq \Z^n\colon [\Z^n\colon R] < X\right\}}.
		\]
		\begin{enumerate}
			\item If $k \in \{1,2,3\}$  and $n \ge 6$, then $p^R_{n,k} = 0$.
			\item If $n \ge 7$, then $p^R_{n,4} = 0$.
			\item If $k \le (6-4\sqrt{2})n + 2\sqrt{2}-\frac{8}{3}$ and $n \ge 7$, then $p^R_{n,k} = 0$. 
		\end{enumerate}
	\end{thm} 
	\noindent For example, while approximately $85\%$ of sublattices of $\Z^7$ are cocyclic, the proportion of subrings of $\Z^7$ that are cocyclic is $0$.
	
		We prove Theorem \ref{proportion_corank_k} by studying the asymptotic rate of growth of the function
	\[
	H_{n,k}(X) = \#\left\{\text{Subrings }  R \subseteq \Z^n\colon [\Z^n\colon R] < X \text{ and } R \text{ has corank at most } k\right\}
	\]
	and comparing it to the asymptotic growth rate of the function 
	\[
	N_n(X) = \#\left\{\text{Subrings }  R \subseteq \Z^n\colon [\Z^n\colon R] < X\right\}.
	\]
	We use techniques from the theory of zeta functions of rings to prove upper bounds for the growth rate of $H_{n,k}(X)$ and compare these results to lower bounds of Isham for $N_n(X)$ \cite[Theorem 1.6]{ish_subrings}.  
	
	A key part of our argument involves counting special classes of matrices in Hermite normal form.  Every finite index sublattice $\Lambda \subseteq \Z^n$ is the column span of a unique $n \times n$ matrix $H(\Lambda)$ in Hermite normal form.  Moreover, $\Z^n/\Lambda$ is isomorphic to the cokernel of this matrix.  We discuss matrices in Hermite normal form and their cokernels at the start of Section \ref{sec:counting} and in Proposition \ref{SNF_prop}.  
    
    Theorem \ref{thm:lattice_corank} says that the probability $\Z^n/\Lambda$ has rank at most $k$ is what one would expect if these random groups were distributed according to the Cohen-Lenstra heuristics.  In \cite{CKK}, the authors also show that Sylow $p$-subgroups of $\Z^n/\Lambda$ are distributed according to the distribution $P_n(G)$ from \eqref{eqn:PnG}.  These results can be interpreted by saying that cokernels of random integer matrices in Hermite normal form follow the Cohen-Lenstra heuristics.  There is a general philosophy that cokernels of families of random integer and $p$-adic matrices should be distributed according to the Cohen-Lenstra heuristics, except when there is additional algebraic structure that must be taken into account; see for example \cite[Section 3.2]{WoodICM}.  Theorem \ref{proportion_corank_k} says that once $n$ is not too small, cokernels of $n \times n$ matrices in Hermite normal form whose columns span a subring of $\Z^n$ are not distributed according to the Cohen-Lenstra heuristics.  The additional algebraic structure of having a column span that is closed under componentwise multiplication leads to a completely different looking distribution of finite abelian groups.

In order to prove Theorem \ref{proportion_corank_k} we need only give an upper bound on the growth rate of $H_{n,k}(X)$, but for small values of $k$ we can prove something much more precise.  We give asymptotic formulas for $H_{n,k}(X)$ when $k \le 3$.
	\begin{thm}\label{Hnk_123}
		Suppose $k \in \{1,2,3\}$ and $n > k$.  There exists a positive real number $C_{n,k}$ so that  
		\[
		H_{n,k}(X) \sim C_{n,k} X (\log X)^{\binom{n}{2}-1}
		\]
		as $X \rightarrow \infty$.
	\end{thm}
	\noindent We give a more precise description of $C_{n,k}$ in Section \ref{sec_corank}.  We discuss the growth of $H_{n,4}(X)$ at the end of that section.  In Section \ref{corank_smalln} we use our knowledge of $N_n(X)$ when $n \le 5$ to compute $p_{n,k}^R$ for all $k$ when $n \le 4$.  We prove that $p_{n,n-1} > 0$ for all $n$, and as a consequence of these results, we show that $p_{5,k} > 0$ for each $k \in \{1,2,3,4\}$.
Finally, an appendix by Gautam Chinta defines a multivariate zeta function  which encodes not only the coranks of subrings of $\Z^n$, but also the full cotypes.  A conjecture explicitly describing this {\em cotype zeta function} for $\Z^4$ is presented.
		
	\section{Counting lattices and counting subrings}\label{sec:counting}
	
	In this paper, we study the functions $H_{n,k}(X)$ and $N_n(X)$ by studying analytic properties of certain zeta functions associated to $\Z^n$.  We begin by introducing the subgroup zeta function of $\Z^n$.
	\begin{defn} \label{defn:zetaZn}
		Let $a_k(\Z^n)$ denote the number of sublattices $\Lambda \subseteq \Z^n$ with $[\Z^n\colon \Lambda] = k$.  Define
		\[
		\zeta_{\Z^n}(s) = \sum_{k=1}^\infty a_k(\Z^n) k^{-s},
		\]
		where $s$ is a complex variable.
	\end{defn}
	Since $\Z^n$ is a finitely generated nilpotent group, this zeta function has an Euler product and we can write
	\[
	\zeta_{\Z^n}(s) = \prod_p \zeta_{\Z^n,p}(s),\ \text{ where }\ \ \  \zeta_{\Z^n,p}(s) = \sum_{e=0}^\infty a_{p^e}(\Z^n) p^{-es}.
	\]
	Let $M_n(\Z)$ denote the set of $n \times n$ matrices with entries in $\Z$.  An invertible matrix $A \in M_n(\Z)$ with entries $a_{ij}$ is in \emph{Hermite normal form} if:
	\begin{enumerate}
		\item $A$ is upper triangular, and 
		\item $0 \le a_{ij} < a_{ii}$ for $1 \le i<j \le n$.
	\end{enumerate}
	Let $\HH_n(\Z)$ denote the set of invertible matrices $A\in M_n(\Z)$ that are in Hermite normal form.  Every sublattice $\Lambda\subseteq \Z^n$ is the column span of a unique matrix $H(\Lambda) \in \HH_n(\Z)$, and moreover, $[\Z^n \colon \Lambda] = \det(H(\Lambda))$.  Counting matrices in $\HH_n(\Z)$ with given determinant proves that 
	\begin{equation}\label{zeta_Zn}
		\zeta_{\Z^n}(s) = \zeta(s) \zeta(s-1)\cdots \zeta(s-(n-1)).
	\end{equation}
	See the book of Lubotzky and Segal for five proofs of this result \cite{LubotzkySegal}.  For an extensive introduction to this topic see the survey of Voll \cite{Voll} or the book of du Sautoy and Woodward~\cite{dusautoy}.
	
	Applying a standard Tauberian theorem allows one to deduce an asymptotic formula for $\sum_{k < X} a_{\Z^n}(k)$ in terms of analytic properties of $\zeta_{\Z^n}(s)$. The following Tauberian theorem is due to Delange. See \cite[Ch III, pages 121-122]{narkiewicz1984number} for an English translation.
	\begin{thm}\label{tauberian_theorem} \cite{delange}
		Let $F(s) = \sum_{n \ge 1} a(n) n^{-s}$ be a Dirichlet series with nonnegative coefficients that converges for $\Re(s) > \alpha > 0.$ If
  \begin{enumerate}
      \item $F(s)$ is analytic on $\Re(s) = \alpha$ except for $s=\alpha$ and 
      \item for $s \sim \alpha$ with $\Re(s) > \alpha$, 
      \[
      F(s) = \frac{G(s)}{(s-\alpha)^\beta}+ H(s)
      \]
      where $G(s)$ and $H(s)$ are analytic at $s= \alpha$ with $G(\alpha) \ne 0$
  \end{enumerate}
  then 
  $$
  \sum_{n \le X} a(n) \sim \frac{G(\alpha)}{\alpha \Gamma(\beta)} X^{\alpha} (\log X)^{\beta -1}.
  $$
	\end{thm}

See the recent survey of Alberts \cite{alberts} for several concrete examples of how Tauberian theorems are applied to Dirichlet series constructed from Euler products.
 
	The function on the right-hand side of equation \eqref{zeta_Zn} has its right-most pole at $s=n$, and this pole is a simple pole.  It is not difficult to see that this function satisfies the hypotheses of Theorem \ref{tauberian_theorem}, and then to carry out the application of this theorem. This leads to the following result.
	\begin{cor}
		Let $n$ be a positive integer $n$ and define
		\[
		L_n(X) = \#\{\text{Sublattices } \Lambda \subseteq \Z^n\colon [\Z^n \colon \Lambda] < X\}.
		\]
		We have
		\[
		L_n(X) \sim \frac{\zeta(n) \zeta(n-1)\cdots \zeta(2)}{n} X^n,
		\]
		as $X \rightarrow \infty$.
	\end{cor}
	
	A main idea of this paper is to follow a similar strategy for the subring zeta function of~$\Z^n$.
	\begin{defn}\label{defn:subring_zeta}
		Let $f_n(k)$ denote the number of subrings $R \subseteq \Z^n$ with $[\Z^n \colon R] = k$.  The \emph{subring zeta function of $\Z^n$} is defined by
		\[
		\zeta_{\Z^n}^R(s) = \sum_{k =1}^\infty f_n(k) k^{-s}.
		\]
	\end{defn}
	Just as $\zeta_{\Z^n}(s)$ has an Euler product, we have 
	\[
	\zeta_{\Z^n}^R(s) = \prod_p \zeta_{\Z^n,p}^R(s),\ \text{ where }\ \ \  \zeta_{\Z^n,p}^R(s) = \sum_{e=0}^\infty f_n(p^e) p^{-es}.
	\]
	
	The subring zeta function of $\Z^n$ is only known explicitly for $n \le 4$.
	\begin{thm}\label{GeneratingFunctions}
		We have 
		\begin{eqnarray*}
			\zeta_{\Z^2}^R(s) & = & \zeta(s), \\
			\zeta_{\Z^3}^R(s) & = & \frac{\zeta(3s-1) \zeta(s)^3}{\zeta(2s)^2}, \\
			\zeta_{\Z^4}^R(s) & = & \prod_p
			\frac{1}{(1-p^{-s})^2(1-p^2 p^{-4s})(1-p^3 p^{-6s})} \Big(1+4 p^{-s}+2 p^{-2s}\\
			& & +(4p-3) p^{-3s}+(5p-1)p^{-4s}  +(p^2-5p)p^{-5s}+(3p^2-4p) p^{-6s} \\ 
			& & -2p^2 p^{-7s}-4p^2 p^{-8s} - p^2 p^{-9s}{\Big)}.
		\end{eqnarray*}
	\end{thm}
	\noindent The statement for $\Z^3$ follows from work of Datskovsky and Wright \cite{DW}, and the statement for $\Z^4$ follows from work of Nakagawa \cite{Nakagawa}.
	
	The right-most pole of $\zeta_{\Z^3}^R(s)$ is located at $s=1$ and has order $3$, and the right-most pole of $\zeta_{\Z^4}^R(s)$ is located at $s=1$ and has order $6$.  Du Sautoy and Grunewald show that for any $n \ge 2$, the zeta function $\zeta_{\Z^n}^R(s)$ satisfies condition (1) of Theorem \ref{tauberian_theorem} \cite[Theorem 1.7]{dusautoy_grunewald}.  Therefore, one can carry out the application of Theorem \ref{tauberian_theorem} for $n \le 4$, leading to the following result.
	\begin{cor}\label{Nn_growth}
		We have 
		\begin{eqnarray*}
			N_2(X) & \sim & X,\\
			N_3(X) & \sim & \frac{1}{2\zeta(2)} X (\log X)^2,\\
			N_4(X) & \sim &\frac{1}{5! \zeta(2)^3}X  (\log  X)^5.
		\end{eqnarray*}
	\end{cor}
 
	For larger $n$ precise asymptotic formulas for $N_n(X)$ are not known. Kaplan, Marcinek, and Takloo-Bighash \cite{kmt} compute the order of growth for $n=5$  but are unable to determine the constant in the asymptotic.
	\begin{thm}\cite[Theorem 6]{kmt}\label{N5X}
		There exists a positive real number $C_5$ such that 
		\[
		N_5(X) \sim C_5 X (\log X)^9
		\]
		as $X \rightarrow \infty$.
	\end{thm}
	
	In the proof of Theorem \ref{proportion_corank_k}, we apply a lower bound for the growth of $N_n(X)$ for $n \ge 7$.  Kaplan, Marcinek, and Takloo-Bighash show that for any $n$, there exists a positive real number $C_n$ such that for $X$ sufficiently large $N_n(X) > C_n X (\log X)^{\binom{n}{2}-1}$ \cite[Theorem 6]{kmt}.  Building on work of Brakenhoff \cite{brakenhoff}, Isham proves a lower bound that is much stronger for $n \ge 7$.
	\begin{thm}\cite[Theorem 1.6 and Proposition 5.4]{ish_subrings}\label{ish_lower}
		Let
		\[
		a(n) =  \max_{0 \leq d \leq n-1} \left(\frac{d(n-1-d)}{n-1+d} + \frac{1}{n-1+d}\right)
		\]
		where $\max_{0 \leq d \leq n-1}$ is a maximum over integers $d \in [0, n-1]$.
		\begin{enumerate}
			\item We have $a(n) \ge (3-2\sqrt{2})(n-1) -(\sqrt{2}-1)$.
			
			\item There exists a positive real number $C_n$ such that for all sufficiently large $X$,
			\[
			N_n(X)  > C_n X^{a(n)} > C_n X^{(3-2\sqrt{2})(n-1) -(\sqrt{2}-1)}.
			\]
			In particular, when $n \ge 7$ there exists a positive real number $C_n$ such that for all sufficiently large $X$,
			\[
			N_n(X) > C_n X^{9/8}.
			\]
		\end{enumerate}
	\end{thm}
	
	Theorem \ref{proportion_corank_k} (1) for $n \ge 7$ follows from Theorems \ref{Hnk_123} and \ref{ish_lower}. We prove Theorem \ref{proportion_corank_k} (2) and the $n=6$ case of Theorem \ref{proportion_corank_k} (1) at the very end of Section \ref{sec_corank}. 
 Theorem \ref{proportion_corank_k} (3) follows from Theorem \ref{ish_lower} and the following result. 
	\begin{thm}\label{Hnk_upper}
		Let $a(n)$ be defined as in Theorem \ref{ish_lower}.  Suppose $n \ge 7$ and $k \le (6-4\sqrt{2})n + 2\sqrt{2}-\frac{8}{3}$.  Then,
		\[
		\lim_{X \rightarrow \infty} \frac{H_{n,k}(X)}{X^{a(n)}} = 0.
		\]
	\end{thm}
\noindent In order to prove this theorem, we give an upper bound on the growth of $H_{n,k}(X)$.  We do this in Corollary \ref{Hnk_upper_cor}.  
	
	\section{The corank at most $k$ zeta function of $\Z^n$}\label{sec_corank}
	
	In order to prove Theorems \ref{proportion_corank_k}, \ref{Hnk_123}, and \ref{Hnk_upper} we introduce the \emph{corank at most $k$ zeta function of $\Z^n$} and study its analytic properties.
    
	\begin{defn}\label{defn:corank_at_most}
		Let $\tilde{h}_{n,k}(j)$ be the number of subrings $R \subseteq \Z^n$ that have corank at most $k$ and index $j$ and ${h}_{n,k}(j)$ be the number of these subrings that have corank exactly $k$.  Clearly, $$\tilde{h}_{n,k}(j) = \sum_{i = 0}^k {h}_{n,i}(j)\; \textrm{   and   } \; H_{n,k}(X) = \sum_{j= 1}^{\lfloor X\rfloor} \tilde{h}_{n,k}(j).$$
		
		We define the \emph{corank at most $k$ zeta function of $\Z^n$} by
		\[
		\zeta_{\Z^n}^{R,(k)}(s) = \sum_{j = 1}^\infty \tilde{h}_{n,k}(j) j^{-s}.
		\]
	\end{defn}
	
    A finite abelian group has rank at most $k$ if and only if each of its Sylow $p$-subgroups has rank at most $k$. For the same reason that the subring zeta function $\zeta_{\Z^n}^R(s)$ has an Euler product, we see that $\zeta_{\Z^n}^{R,(k)}(s)$ has an Euler product as well.  We have
	\[
	\zeta_{\Z^n}^{R,(k)}(s) = \prod_p \zeta_{\Z^n, p}^{R,(k)}(s),\ \text{ where }\ \ \ \zeta_{\Z^n, p}^{R,(k)}(s) = \sum_{e =0}^\infty \tilde{h}_{n,k}(p^e)p^{-es}.
	\]
	
 In order to prove Theorem \ref{Hnk_123}, we find the right-most pole of $\zeta_{\Z^n}^{R, (k)}(s)$ for $k \le 3$ and then apply Theorem \ref{tauberian_theorem}. In this section, we prove Theorem \ref{Hnk_123}, first for $k=1$, then $k=2$, and finally for $k=3$.

	\begin{prop}\label{count_cocyclic}
		Suppose $n \ge 2$ and $e \ge 1$.  Then $h_{n,1}(p^e) = \binom{n}{2}$.
	\end{prop}
	\noindent We prove this result in Section \ref{subring_matrices}.  We note that this is closely related to a result of Brakenhoff \cite[Theorem 1.5]{brakenhoff}.  
 
 Proposition \ref{count_cocyclic} leads directly to the following expression for the corank at most $1$ zeta function of $\Z^n$.
	\begin{prop}\label{cocyclic_zeta}
		Suppose $n \ge 2$.  We have 
		\[
		\zeta_{\Z^n}^{R,(1)}(s)  = \zeta(s) \prod_p \left(1 + \left(\binom{n}{2} - 1\right)p^{-s}\right).
		\]
	\end{prop}
 
The $k=1$ case of Theorem \ref{Hnk_123} follows directly from Proposition \ref{cocyclic_zeta}. 
The following is an instance of the Factorization Method discussed in the survey of Alberts \cite{alberts}.

 \begin{cor} \label{cor:cocyclic_asymptotic}
     Let $m = \binom{n}{2}$. We have
     \[
     H_{n,1}(X) \sim \frac{1}{(m-1)!} \prod_p (p^{-m} (p-1)^{m-1} (p+(m-1)) X (\log X)^{m-1}.
     \]
 \end{cor}
  \noindent Note that for $n = 2$ this is consistent with the result from Corollary \ref{Nn_growth} that $N_2(X) \sim X$.
  
 \begin{proof}
   The right-most pole of $\zeta_{\Z^n}^{R,(1)}(s)$ is located at $s = 1$. We will prove that the pole has order $\binom{n}{2}$. 
  Multiplying the zeta function by $\frac{1}{\zeta(s)^t}$ for some integer $t > 0$, observe that
  \begin{align*}
    \zeta_{\Z^n}^{R,(1)}(s) \prod_p (1-p^{-s})^{t}&=\prod_p \left(\left(1 + \left(\binom{n}{2} - 1\right)p^{-s}\right)(1-p^{-s})^{t-1}\right)\\
    &=\prod_p\left(\left(1 + \left(\binom{n}{2} - 1\right)p^{-s}\right) \sum_{k=0}^{t-1} (-1)^k\binom{t-1}{k} p^{-ks}\right).
  \end{align*} 
  From here it is clear that $\zeta_{\Z^n}^{R, (1)}(s) \prod_p (1-p^{-s})^{t}$ still has a pole at $s=1$ when $t < \binom{n}{2}$, and does not have a pole at $s=1$ when $t = \binom{n}{2}.$
  
  We now apply Theorem \ref{tauberian_theorem}.
  To find $C_{n,1}$, we evaluate $\zeta(s)^{\binom{n}{2}} \zeta_{\Z^n}^{R,(1)}(s)$ at $s=1$ and simplify.
 \end{proof}
 We will apply this result in Section \ref{corank_smalln} when we study the proportion of subrings of $\Z^n$ with small corank for $n \le 5$.

	In order to state the analogous results for subrings of corank $2$ and corank $3$ we first recall some material due to Liu about irreducible subrings of $\Z^n$ \cite{Liu}.  Liu defines irreducible subrings of $\Z_p^n$, but for notational convenience we prefer to define the corresponding notion for subrings of $\Z^n$ with index equal to a power of $p$.
	\begin{defn}
		A subring $R \subseteq \Z^n$ with index equal to a power of $p$ is \emph{irreducible} if for each $(x_1, \ldots, x_n) \in R,\ x_1\equiv x_2\equiv \cdots \equiv x_n \pmod{p}$.
	\end{defn}
	
	The motivation for this definition comes from the following decomposition result of Liu.
	\begin{thm}\cite[Theorem 3.4]{Liu}
		A subring $R \subseteq \Z^n$ with index equal to a power of $p$ can be written uniquely as a direct sum of irreducible subrings.
	\end{thm}
	Liu uses this result to prove a recurrence for subrings of $\Z^n$ in terms of irreducible subrings of $\Z^n$ and subrings of $\Z^j$ for $j< n$ \cite[Proposition 2.8]{Liu}.  Let $g_n(k)$ be the number of irreducible subrings of $\Z^n$ of index $k$.  We note that Liu uses slightly different notation for counting irreducible subrings.  Our $g_n(k)$ is denoted by $g_{n-1}(k)$ in \cite{Liu}.
	
	We now state the analogues of Proposition \ref{count_cocyclic} for subrings of corank $2$ and $3$.
	\begin{thm}\label{corank2_upper}
		Suppose $e \ge 2$ and $n \ge 3$. Then
		\[
		h_{n, 2}(p^e) =\frac{(3n^2-17n+36)}{12} \binom{n-1}{2} g_3(p^e) + 3\binom{n-1}{3}(e-1).
		\]
	\end{thm}
	We now give a similar but more complicated, result for subrings of corank $3$.
  \begin{thm}\label{corank3_upper}
		Suppose $e \ge 3$ and $n \ge 4$. Then
		\[
	 	h_{n,3}(p^e) = \frac{n^3 - 11n^2 + 40n - 40}{8} \binom{n-1}{3} g_4(p^e) + (3n - 5) \binom{n-1}{4}\sum_{j=2}^{e-1} (j-1) g_3(p^j).
	 	\]
   \end{thm}
	\noindent We will prove these two theorems in Section \ref{subring_matrices}.  

 Our aim is to prove Theorem \ref{Hnk_123} using the explicit formulas for $h_{n,k}(p^e)$ when $k \in \{1,2,3\}.$ 

 \begin{prop} \label{prop:corank2_zeta} 
 Let 
  \begin{align*}
  a(n) &= \frac{3n^2-17n+36}{12} \binom{n-1}{2}\\
  b(n) &= 3\binom{n-1}{3}\\
m &= \binom{n}{2}.
  \end{align*}
 
  We have
      \begin{align*}
     \zeta_{\Z^n}^{R, (2)}(s)
     &= \zeta(s)^2\zeta(3s-1)\prod_p 
    \bigg(1   + \left(m-2\right) p^{-s} + ( a(n) + b(n) - m+1) p^{-2 s}
     - a(n)p^{-3s} \\
     &+ (  a(n)-1 )p^{1-3s} + (a(n)  - m + 2)p^{1-4 s} -
 (2a(n) + b(n)  - m+1 )p^{1-5s}  \bigg).
  \end{align*}
  The right-most pole of this function occurs at $s=1$ and has order $m$.
 \end{prop}
 \noindent We defined $a(n)$ and $b(n)$ in this way so that Theorem \ref{corank2_upper} becomes 
 \[
 h_{n,2}(p^e) = a(n) g_3(p^3) + b(n) (e-1).
\]
 Before giving the proof, we remark that we can explicitly determine the constant derived from applying Theorem \ref{tauberian_theorem} to the function in Proposition \ref{prop:corank2_zeta}.

\begin{cor}\label{cor:corank2_constant}
    Let $a(n), b(n),$ and $m$ be defined as in Proposition \ref{prop:corank2_zeta}. Then we have
\[
    H_{n,2}(X) \sim C_{n,2} X(\log X)^{m-1}
\]
    where
 \begin{align*}
     C_{n,2} &= \frac{\zeta(2)}{(m-1)!} \prod_p (1-p^{-1})^{m-2} \bigg(1  + \left(m-2\right) p^{-1} + ( 2a(n) + b(n) - m) p^{-2 } \\
     &-  ( m-2)p^{-3} -(2a(n) + b(n)  - m+1 )
 p^{-4}  \bigg).
  \end{align*}
\end{cor}

Note that $H_{3,2}(X) = N_3(X)$ since subrings in $\Z^3$ have corank at most 2. The constant $C_{3,2}$ from Corollary \ref{cor:corank2_constant} is consistent with the constant in the asymptotic for $N_3(X)$ given in Corollary \ref{Nn_growth}. 

  \begin{proof}[Proof of Proposition \ref{prop:corank2_zeta}]

  Consider $\tilde{h}_{n,2}(p^e)$. Observe that $\tilde{h}_{n,2}(p^0) = 1$ and $\tilde{h}_{n,2}(p^1) = \binom{n}{2}$. For any $e \ge 2$, we have that
  $$
  \tilde{h}_{n,2}(p^e) = \binom{n}{2} + h_{n,2}(p^e).
  $$
  Applying Theorem \ref{corank2_upper}, we obtain
  \begin{align*}
     \zeta_{\Z^n,p}^{R, (2)}(s) &= 1 + \sum_{e \ge 1}\binom{n}{2} p^{-es} + \sum_{e \ge 2} \left(  a(n)g_3(p^e) + b(n) (e-1)\right)p^{-es}.
  \end{align*}

Liu shows that
	\begin{align*}
		\sum_{e \ge 0} g_3(p^e) p^{-es}&=\sum_{e \ge 2} g_3(p^e) p^{-es}= \frac{p^{-2s} + p^{1-3s} + 2p^{1-4s}}{(1-p^{-s})(1-p^{1-3s})}.
	\end{align*}
This is $B_2(p,x)$ with $x = p^{-s}$ \cite[page 296]{Liu}.

 Using this formula and standard geometric series formulas, we simplify our expression as follows using Mathematica \cite{mathematica}. The first author has posted the Mathematica worksheets on her website \cite{mathematica_work}. We have
  \begin{align*}
     \zeta_{\Z^n,p}^{R, (2)}(s)
     &= 1 + \binom{n}{2}  \frac{p^{-s}}{1-p^{-s}} + a(n) \frac{p^{-2s} + p^{1-3s} + 2p^{1-4s}}{(1-p^{-s})(1-p^{1-3s})} + b(n) \frac{p^{-2s}}{(1-p^{-s})^2}\\
     &= \frac{1}{(1-p^{-s})^2 (1-p^{1-3s})} \bigg(1  + ( a(n) + b(n) - \binom{n}{2}+1) p^{-2 s} + \left(\binom{n}{2}-2\right) p^{-s} \\
     &- a(n)p^{-3s}  + p^{1-3s}(  a(n)-1 ) + p^{1-4 s} (a(n)  - \binom{n}{2} + 2)-
 p^{1-5s} (2a(n) + b(n)  - \binom{n}{2}+1 ) \bigg).
  \end{align*}
 
An argument like the one given in the proof of Corollary \ref{cor:cocyclic_asymptotic} shows that the right-most pole of this function is located at $s=1$, and the order of this pole is $\binom{n}{2}$.  Applying Theorem \ref{tauberian_theorem} now gives the asymptotic result. 
\end{proof}

\begin{prop}\label{prop:corank3_zeta}
        Let 
       \begin{align*}
  a(n) &= \frac{3n^2-17n+36}{12} \binom{n-1}{2}\\
  b(n) &= 3\binom{n-1}{3}\\
  c(n) &= \frac{n^3-11n^2+40n-40}{8} \binom{n-1}{3}\\
  d(n) &= (3n-5) \binom{n-1}{4}\\
  m &= \binom{n}{2}.
  \end{align*}
The right-most pole of $\zeta_{\Z^n}^{R,(3)}(s)$ occurs at $s=1$ and has order $m$.  We have 
 \[
    H_{n,3}(X) \sim C_{n,3} X(\log X)^{m-1}
\]
where
\begin{align*}
    C_{n,3} = \frac{1}{(m-1)!} \prod_p (1-p^{-1})^{m-4} &\bigg(1 + p^{-1} (m-4) + p^{-2} (6 + 2 a(n) + b(n) +c(n) - 3 m) \\
    &+p^{-3} (-4 - 4 a(n) - 2 b(n)+ 6 c(n) + 
  3 d(n) + 3 m)\\
  &+ p^{-4}(1 + 2a(n) + b(n) - 7 c(n) + 2 d(n) - m)\bigg).
\end{align*}
\end{prop}
 \noindent We defined $c(n)$ and $d(n)$ in this way so that Theorem \ref{corank3_upper} says 
 \[
 h_{n,3}(p^e) = c(n) g_4(p^e) + d(n) \sum_{j=2}^{e-1} g_3(p^j).
\]
We could write down an expression for $\zeta_{\Z^n}^{R,(3)}(s)$ analogous to the one given for $\zeta_{\Z^n}^{R,(2)}(s)$ in Proposition \ref{prop:corank2_zeta}, but in the interest of space, we omit it.

Note that constant $C_{4,3}$ from Proposition \ref{prop:corank3_zeta} is consistent with the constant in the asymptotic for $N_4(X)$ given in Corollary \ref{Nn_growth}. See the first author's website \cite{mathematica_work} for details. 

\begin{proof}
  Consider $\tilde{h}_{n,3}(p^e)$. Observe that $\tilde{h}_{n,3}(p^0) = 1,\ \tilde{h}_{n,3}(p^1) = \binom{n}{2}$, and $\tilde{h}_{n,3}(p^2) = \tilde{h}_{n,2}(p^2)$.
  For $e \ge 3$, 
  \[
  \tilde{h}_{n,3}(p^e) = \binom{n}{2} + h_{n,2}(p^e) + h_{n,3}(p^e).
  \]
Plugging these expressions into $\zeta_{\Z^n, p}^{R, (3)}(s)$, we have
  \begin{eqnarray}\label{corank3_zeta}
      \zeta_{\Z^n, p}^{R, (3)}(s) &=& 1 + \sum_{e\ge 1} \binom{n}{2} p^{-es} + \sum_{e \ge 2} \left(a(n) g_3(p^3) + b(n)(e-1)\right)p^{-es} \\
      &+& \sum_{e \ge 3} \left(c(n) g_4(p^e) + d(n) \sum_{j=2}^{e-1} (j-1)g_3(p^j)\right)p^{-es},\nonumber
  \end{eqnarray}
where $a(n), b(n), c(n)$, and $d(n)$ are defined as in the statement of the proposition.

The proofs of the $k=1$ and $k=2$ cases of Theorem \ref{Hnk_123} handle
the first three summands, so we focus on the last summand. Many of these calculations and simplifications were performed in Mathematica \cite{mathematica}; see the first author's website \cite{mathematica_work} for the code.

  First consider
  \[
  \sum_{e\ge 3} c(n) g_4(p^e)p^{-es}.
  \]
Liu proves that 
   \begin{align*}
 \sum_{e\ge 3} g_4(p^e)p^{-es} &= \frac{p^{-3s}}{(1-p^{-s})^2 (1-p^{1-3s}) (1-p^{2-4s}) (1-p^{3-6s})}\bigg( 1 + (p^2+p-1)p^{-s} \\
  &+ (5p^2-p)p^{-2s} 
  +(p^3+p^2-p)p^{-3s} + (7p^3-11p^2+p)p^{-4s} + (p^3+p^2)p^{-5s}\\
  &+ (3p^4 -13p^3+3p^2)p^{-6s} + (-p^5+2p^3)p^{-7s} + (-4p^5 -6p^4+4p^3)p^{-8s} \\
  &+ (-2p^5+p^3)p^{-9s} +(-3p^6+4p^5)p^{-10s} +6p^{6-12s} \bigg).
  \end{align*}
  \noindent This is $B_3(p,x)$ with $x = p^{-s}$ \cite[Proposition 6.3]{Liu}.

  Next consider 
  \begin{align*}
  \sum_{e \ge 3} \left( \sum_{j=2}^{e-1} (j-1)g_3(p^j)\right)p^{-es} &= \sum_{e \ge 3}\left( \sum_{j=2}^{e-1} (j-1)g_3(p^j) p^{-js} p^{(-e+j)s}\right)\\
  &= g_3(p^2)p^{-2s} \sum_{e \ge 1} p^{-es} + 2g_3(p^3)p^{-3s} \sum_{e \ge 1} p^{-es}+ 3g_3(p^4) p^{-4s}\sum_{e \ge 1}p^{-es} + \cdots \\
  &= \frac{p^{-s}}{1-p^{-s}} \sum_{e \ge 2} (e-1)g_3(p^{e}) p^{-es}.
  \end{align*}
We consider the derivative with respect to $x$ of both expressions for $B_2(p,x)$ given by \cite[page 296]{Liu}.  Observe that 
  \[
  \frac{d}{dx} \left(\sum_{e \ge 2} g_3(p^e)x^e \right) =x^{-1} \sum_{e \ge 2} eg_3(p^e) x^e.
  \]
  Differentiating the rational function expression for $B_2(p,x)$ with respect to $x$ gives
  \begin{align*}
  \frac{d}{dx} \left(\sum_{e \ge 2} g_3(p^e) x^e \right) &= \frac{d}{dx} \left( \frac{x^2 + p x^3 + 2p x^4}{(1-x)(1-p x^3)} \right)\\
  &=
\frac{-x (-2 + x - 3 p x - 6 p x^2 + 5 p x^3 + 2 p x^4 + 3 p^2 x^5)}{(1 - x)^2 (1 - p x^3)^2}.
\end{align*}

Combining the expressions for each of the four summands in \eqref{corank3_zeta} and putting them over a common denominator gives the formula for $\zeta_{\Z^n}^{R, (3)}(s)$.

   From here, we can see that neither of these last two summands contributes a pole to the right of $s=1$ when we take the Euler product. Since $g_4(p^e) \le f_4(p^e)$, we see that
   \[
\prod_p \left(1+\sum_{e \ge 2} g_4(p^e)p^{-es}\right)
   \]
converges to the right of $\Re(s) = 1$ and cannot have a pole at $s=1$ of order larger than $\binom{n}{2}$. An analogous statement holds for 
\[
\prod_p \left(1+\sum_{e\ge 3} \left(\sum_{j=2}^{e-1} (j-1)g_3(p^j)\right) p^{-es}\right).
\]

We can now apply the strategy of the proof of Corollary \ref{cor:cocyclic_asymptotic}, multiplying the expression for $\zeta_{\Z^n, p}^{R, (3)}(s)$ by $(1-p^{-s})^{\binom{n}{2}}$ and then using the fact that if $a_n>0$ for all $n$, then the infinite product $\prod_n (1+a_n)$ converges if and only if $\sum_n a_n$ converges.  In this way, we see that $\zeta_{\Z^n}^{R,(3)}(s)$ has its right-most pole at $s=1$ and that pole has order exactly $\binom{n}{2}$.

  \end{proof}

We have now completed the proof of Theorem \ref{Hnk_123}.
\begin{proof}[Proof of Theorem \ref{Hnk_123}]
This theorem follows directly from Corollary \ref{cor:cocyclic_asymptotic}, Corollary \ref{cor:corank2_constant}, and Proposition \ref{prop:corank3_zeta}.
\end{proof}
We can use the expression for $\tilde{h}_{4,3}(p^e)$ that comes from Proposition \ref{count_cocyclic}, Theorem \ref{corank2_upper}, and Theorem \ref{corank3_upper} along with the ideas in this proof to verify that $\zeta_{\Z^4}^{R,(3)}(s) = \zeta_{\Z^4}^R(s)$, which gives a nice check of our results.  Similarly, we can use the expression for $\tilde{h}_{3,2}(p^e)$ that comes from Theorem \ref{corank2_upper} and Proposition \ref{count_cocyclic} to verify that $\zeta_{\Z^3}^{R,(2)}(s) = \zeta_{\Z^3}^R(s)$.

		A main idea of the proof of Theorem \ref{Hnk_upper} (and thus Theorem \ref{proportion_corank_k}) is to prove upper bounds for $\tilde{h}_{n,k}(p^e)$, which then imply that the right-most pole of $\zeta_{\Z^n}^{R,(k)}(s)$ cannot be too large.  Applying Theorem \ref{tauberian_theorem} then completes the proof. We state a general upper bound for the number of subrings of $\Z^n$ of corank $k$ and index $p^e$.  We defer the proof until the next section. 
	\begin{thm}\label{general_upper}
Suppose $e,n$ and $k$ are positive integers with $k \le n-1$.
\begin{enumerate}
\item We have
\[
\binom{n-1}{k} g_{k+1}(p^e) \le h_{n,k}(p^e) \le (n-k)^k \binom{n-1}{k} g_{k+1}(p^e).
\]
\item We have 
\[
\tilde{h}_{n,k}(p^e) \le k (n-1)^k 2^{n-1} f_{k+1}(p^e).
\]
\end{enumerate}
	\end{thm}
\noindent We could certainly prove sharper bounds, but these suffice for our main application.

We apply the second part of Theorem \ref{general_upper} along with an upper bound on $N_{k+1}(X)$ due to Kaplan, Marcinek, and Takloo-Bighash in order to give an upper bound on $H_{n,k}(X)$.
\begin{thm}\label{KMTB_upper}\cite[Theorem 6]{kmt}
Let $k \ge 5$ be a positive integer.  For any $\epsilon > 0$, there exists a constant $C_{k,\epsilon} > 0$ depending on $k$ and $\epsilon$ such that 
\[
N_{k+1}(X) < C_{k,\epsilon} X^{\frac{k}{2} - \frac{2}{3} + \epsilon}
\]
for all $X > 0$. 
\end{thm}

We apply this result to prove the following upper bound for $H_{n,k}(X)$.
\begin{cor}\label{Hnk_upper_cor}
For any $\epsilon > 0$, we have
\[
H_{n,k}(X) = O(X^{\frac{k}{2} - \frac{2}{3} + \epsilon})
\]
where the constant depends on $k,n$, and $\epsilon$.
\end{cor}
\begin{proof}
Our first main goal is to show that if the right-most pole of $\zeta_{\Z^{k+1}}^R(s)$ is at $s = \alpha$, then for any $n,\ \zeta_{\Z^n}^{R,(k)}(s)$ converges whenever $\Re(s) > \alpha$.  We have 
\[
\zeta_{\Z^n}^{R,(k)}(s) = \prod_p \zeta_{\Z^n,p}^{R,(k)}(s) = \prod_p \left(1 + \sum_{e \ge 1} \tilde{h}_{n,k}(p^e) p^{-es}\right).
\]
We recall that  
\[
\prod_p \left(1 + \sum_{e \ge 1} \tilde{h}_{n,k}(p^e) p^{-es}\right)
\]
converges if and only if 
\[
\sum_p \sum_{e \ge 1} \tilde{h}_{n,k}(p^e) p^{-es}
\]
does.  Theorem \ref{general_upper} (2) implies that for any positive real $s$
\[
\sum_p \sum_{e \ge 1} \tilde{h}_{n,k}(p^e) p^{-es} \le 
k(n-1)^k 2^{n-1} \sum_p \sum_{e \ge 1} f_{k+1}(p^e) p^{-es}.
\]
Since $k(n-1)^k 2^{n-1}$ is a constant, this expression converges if and only if $\sum_p \sum_{e \ge 1} f_{k+1}(p^e) p^{-es}$ does.  But this expression converges if and only if 
\[
\prod_p \left(1 + \sum_{e \ge 1} f_{k+1}(p^e) p^{-es}\right) = \zeta_{\Z^{k+1}}^{R}(s)
\]
converges.  

Suppose the right-most pole of $\zeta_{\Z^{k+1}}^R(s)$ is at $s=\alpha$.  Since $\zeta_{\Z^n}^{R,(k)}(s)$ has no poles to the right of $\alpha$, Theorem \ref{tauberian_theorem} now implies that for any $\epsilon > 0,\ H_{n,k}(X) = O(X^{\alpha+\epsilon})$.  Applying Theorem \ref{KMTB_upper} completes the proof.
\end{proof}
	
We now show how this result completes the proof of Theorem \ref{Hnk_upper}.  	
\begin{proof}[Proof of Theorem \ref{Hnk_upper}]
Corollary \ref{Hnk_upper_cor} implies that for any $\epsilon > 0$ there exists a constant $C_{n,k,\epsilon} > 0$ such that
\[
H_{n,k}(X) < C_{n,k,\epsilon} X^{\frac{k}{2}-\frac{2}{3}+\epsilon}.
\]
  
  By Theorem \ref{ish_lower} (1), $X^{a(n)} >X^{(3-2\sqrt{2})(n-1) -(\sqrt{2}-1)}$. We find that
\[
  \lim_{X \rightarrow \infty} \frac{H_{n,k}(X)}{X^{a(n)}}=0
\]  
whenever $\frac{k}{2} - \frac{2}{3} < (3-2\sqrt{2})(n-1) -(\sqrt{2}-1)$. Solving for $k$ completes the proof.
\end{proof}

The proof of Theorem \ref{proportion_corank_k} (2) follows a similar argument, but in place of the upper bound from Corollary \ref{Hnk_upper_cor} we use the result from Theorem \ref{N5X} about the growth of $N_5(X)$.

\begin{proof}[Proof of Theorem \ref{proportion_corank_k} (2)]
Theorem \ref{N5X} implies that there exists a $C_5>0$ such that $N_5(X) \sim C_5 X (\log X)^9$.  The argument from the proof of Corollary \ref{Hnk_upper_cor} then shows that for any $\epsilon>0,\ H_{n,4}(X) = O(X^{1+\epsilon})$, where the constant depends on $n$ and $\epsilon$.  Note that $a(n)>1$ for any $n \ge 7$.  Applying the argument from the proof of Theorem \ref{Hnk_upper} completes the proof.
\end{proof}

We now turn to the statement of Theorem \ref{proportion_corank_k} (1) for the particular case $n = 6$.  While we do not know the asymptotic rate of growth of $N_6(X)$, we do have a lower bound that is good enough for our purposes.  We learned the following fact from Gautam Chinta and Ramin Takloo-Bighash.
\begin{prop}\label{Z6_lower}
We have
\[
\lim_{X \rightarrow \infty} \frac{X(\log X)^{15}}{N_6(X)} = 0.
\]
\end{prop}
\begin{proof}
The key idea is to apply a result for subrings of small index from \cite{akkm}.  By Theorem \ref{tauberian_theorem} it is enough to show that the order of the pole at $s=1$ is at least $\binom{6}{2}+1$.  

We know that $\Z^6$ has $\binom{6}{2}$ subrings of index $p$ and at least $p^6$ subrings of index $p^7$ \cite[Theorem 1.3]{akkm}.  Since
\[
\prod_p (1+ 15 p^{-s}  + p^6 p^{-7s})
\]
has a pole of order at least $16$ at $s=1$, we see that $\zeta_{\Z^6}^R(s)$ does as well.
\end{proof}

\begin{proof}[Proof of Theorem \ref{proportion_corank_k} (1) for $n=6$]
Theorem 1.4 implies that $H_{6,3}(X) = o(X (\log(X)^{15})$, so the result of the previous proposition implies that 
\[
\lim_{X \rightarrow \infty} \frac{H_{6,3}(X)}{N_6(X)} = 0.
\]
\end{proof}
 
We will prove Theorem \ref{general_upper} at the end of the next section.	
	
	\section{Subring matrices and the proofs of Theorems \ref{corank2_upper}, \ref{corank3_upper}, and \ref{general_upper}}\label{subring_matrices}
	
	A main idea that we use to prove upper bounds for $h_{n,k}(p^e)$ is to consider matrices in $\HH_n(\Z)$ whose columns span a subring of $\Z^n$ of corank at most $k$.  Throughout this paper, if $A \in M_n(\Z)$ we write $\col(A)$ for the column span of $A$.  We write $v_1,\ldots, v_n$ for the columns of $A$ and $a_{ij}$ for the entries of $A$.
	\begin{defn}
		A matrix $A \in \HH_n(\Z)$ is a \emph{subring matrix} if 
		\begin{enumerate}
			\item the identity $(1,1,\ldots,1)^T \in \col(A)$, and
			\item for any columns $v_i, v_j$ of $A$, we have $v_i \circ v_j \in \col(A)$.
		\end{enumerate}
		We see that $A$ is a subring matrix if and only if $\col(A)$ is a subring of $\Z^n$.  Moreover, $\det(A) = [\Z^n\colon \col(A)]$.
	\end{defn}
	Suppose $A \in  \HH_n(\Z)$ is a subring matrix.  Since $(1,1,\ldots,1)^T \in \col(A)$, it is clear that $a_{nn} = 1$.  If $\det(A) = p^e$, then the diagonal entries of $A$ are $(p^{e_1}, p^{e_2},\ldots, p^{e_{n-1}},1)$, where each $e_i \ge 0$ and $e_1+\cdots + e_{n-1} = e$.  That is, $(e_1,\ldots, e_{n-1})$ is a weak composition of $e$ into $n-1$ parts.  Let $\alpha = (e_1,\ldots, e_{n-1})$ and define $g_\alpha(p)$ to be the number of irreducible subrings of $\Z^n$ with diagonal entries $(p^{e_1},\ldots, p^{e_{n-1}},1)$.  
	
	We now prove that the corank of a subring $R \subseteq \Z^n$ is at most $n-1$.  Recall that the \emph{cokernel} of a matrix $A \in M_n(\Z)$ is $\cok(A) = \Z^n/\col(A)$. If $A$ is a subring matrix then the corank of $\col(A)$ is the number of nontrivial invariant factors of $\cok(A)$.  We recall some basic facts about the Smith normal form of an integer matrix.
	\begin{prop}\label{SNF_prop}
		Let $A \in M_n(\Z)$ be invertible.  There exist $P, Q \in \GL_n(\Z)$ such that $PAQ = S$ is a diagonal matrix whose diagonal entries $(s_1,s_2,\ldots, s_n)$ are positive integers satisfying $s_i \mid s_{i+1}$ for all $1 \le i \le n-1$.  Since $\cok(A) \cong \cok(PAQ) \cong \cok(S)$, we have 
		\[
		\cok(A) \cong \Z/s_1 \Z \times \Z/s_2 \Z \times \cdots \times \Z/s_n\Z.
		\]
		Moreover, these $s_i$ are uniquely determined and 
		\[
		s_1 \cdots s_i = \gcd(i \times i \text{ minors of } A).
		\]
	\end{prop}
	We have seen that an $n \times n$ subring matrix $A$ has an entry equal to $1$, and therefore the greatest common divisor of the $1 \times 1$ minors of $A$ is $1$.  By Proposition \ref{SNF_prop}, $\cok(A)$ has at most $n-1$ nontrivial invariant factors, and therefore $\col(A)$ has corank at most $n-1$.
	
	\begin{prop}\label{subringZ2}
		We have $\zeta_{\Z^2}^R(s) = \zeta(s)$.  Every proper subring $R \subseteq \Z^2$ has corank $1$.
	\end{prop}
	\begin{proof}
		Suppose $A \in \HH_2(\Z)$ is a subring matrix with $\det(A) = k$. Then $A = \left(\begin{smallmatrix} k & x \\ 0 & 1\end{smallmatrix}\right)$ where $0 \le x < k$.  Since $(1,1)^T \in \col(A)$ we see that $x = 1$.
	\end{proof}
	
	Liu has determined the conditions under which the column span of a subring matrix is an irreducible subring.
	\begin{prop}\cite[Proposition 3.1]{Liu}\label{liu31}
		Suppose $A \in \HH_n(\Z)$ is a subring matrix with columns $v_1,\ldots, v_n$ and determinant equal to a power of $p$.  Then $\col(A)$ is an irreducible subring if and only if $v_n = (1,1,\ldots, 1)^T$ and for each $i \in [1,n-1]$ every entry of $v_i$ is $0$ modulo $p$.
	\end{prop} 
	If $A \in \HH_n(\Z)$ is a subring matrix with last column $(1,1,\ldots, 1)^T$, determinant equal to a power of $p$, and every entry of its first $n-1$ columns is divisible by $p$, then we say that $A$ is an \emph{irreducible subring matrix}.  We see that $g_n(p^e)$ is equal to the number of $n\times n$ irreducible subring matrices with determinant $p^e$.  The diagonal entries of an irreducible subring matrix $A$ with $\det(A) = p^e$ are of the form $(p^{e_1},\ldots, p^{e_{n-1}},1)$ where each $e_i \ge 1$ and $e_1+\cdots + e_{n-1} = e$.  That is, $(e_1,\ldots, e_{n-1})$ is a composition of $e$ into $n-1$ parts. We see that $g_n(p^e)$ is equal to the sum of $g_\alpha(p)$ taken over all compositions $\alpha$ of $e$ into $n-1$ parts.
	
	\begin{prop}\cite[Proposition 3.3]{Liu}\label{liu33}
		Let $A$ be a subring matrix with diagonal entries $(p^{e_1}, \ldots, p^{e_{n-1}}, 1)$.
		If $e_i > 0$ for all $1 \le i \le n-1$, then $A$ is irreducible. 
	\end{prop}
	Liu remarks that Propositions \ref{liu31} and \ref{liu33} give a sufficient condition for determining whether a subring matrix is irreducible. Namely, it suffices to check that the diagonal is of the form $(p^{e_1}, \ldots, p^{e_{n-1}}, 1)$ where each $e_i \ge 1$.  These remarks lead to the following proposition.
	\begin{prop}\label{irred_corank_prop}
		Suppose $R \subseteq \Z^n$ is an irreducible subring.  The corank of $R$ is $n-1$.
	\end{prop}
	\begin{proof}
		Let $A$ be the irreducible subring matrix for which $\col(A) = R$.  Proposition \ref{liu31} implies that the last column of $A$ is $(1,\ldots, 1)^T$.  In every other column of $A$, each entry is divisible by $p$.  This implies that every $2\times 2$ minor of $A$ is divisible by $p$.  By Proposition \ref{SNF_prop}, $\cok(A)$ has exactly $n-1$ nontrivial invariant factors.
	\end{proof}
	
	We next show that in a subring matrix $A$ with a given diagonal certain collections of entries must be divisible by $p$ and a particular submatrix constructed from $A$ is an irreducible subring matrix.

	\begin{prop}\label{prop:subring_matrix_divis}
Let $A$ be an $n \times n$ subring matrix with diagonal $(p^{e_1},p^{e_2},\ldots, p^{e_{n-1}},1)$ and define $I = \{i_1,i_2,\ldots, i_k\}$ where $1 \le i_1<i_2<\cdots < i_k \le n-1$ and $e_j \neq 0$ if and only if $j \in I$.
\begin{enumerate}
\item For any $j_1, j_2$ satisfying $1\le j_1 \le j_2 \le k$, we have $v_{i_{j_1}} \circ v_{i_{j_2}} \in \Span(v_{i_1},v_{i_2},\ldots, v_{i_k})$.
\item Consider the $k \times k$ matrix obtained by first taking the $n \times k$ matrix with columns $v_{i_1},v_{i_2},\ldots, v_{i_k}$ and then deleting the $j$\textsuperscript{th} row for any $j \not \in I$.  Now append a row where every entry is $0$ to the bottom of this matrix, and append a final column where every entry is $1$.  This is a $(k+1) \times (k+1)$ irreducible subring matrix.
\item For any $j_1, j_2$ satisfying $1\le j_1 \le j_2 \le k$, we have $p \mid a_{i_{j_1} i_{j_2}}$.
\item If $i \in I$, then every entry in the column $v_i$ is divisible by $p$.
\end{enumerate}
	\end{prop}
	
\begin{proof}
We prove these statements in order.  
\begin{enumerate}
\item Since $\col(A)$ is multiplicatively closed, for any pair $i_{j_1}, i_{j_2} \in I$ there exist unique $c_1,c_2,\ldots, c_n \in \Z$ such that 
\[
v_{i_{j_1}} \circ v_{i_{j_2}} = \sum_{\ell = 1}^n c_\ell v_\ell.
\]
Suppose $j \not\in I$, so that $p^{e_j} = 1$.  Since $A$ is in Hermite normal form, the only nonzero entry in the $j$\textsuperscript{th} row of $A$ is this $1$ in column $j$.  In particular, the entry of $v_{i_{j_1}} \circ v_{i_{j_2}}$ in row $j$ is $0$.  This implies $c_j = 0$.  We conclude that $v_{i_{j_1}} \circ v_{i_{j_2}}$ is in the span of the columns $v_i$ where $i \in I$.

\item Consider the $n \times k$ matrix with columns $v_{i_1},\ldots, v_{i_k}$.  Part (1) implies that the column span of this matrix is multiplicatively closed.  If $j\not\in I$, then every entry of the $j$\textsuperscript{th} row of this matrix is $0$.  We delete these rows and see that the column span of this $k \times k$ matrix is still multiplicatively closed.  Appending a row where every entry is $0$ to the bottom of the matrix and a final column where every entry is equal to $1$ gives a $(k+1) \times (k+1)$ matrix whose column span is multiplicatively closed.  The diagonal of this matrix consists of positive powers of $p$ in the first $k$ entries and then a $1$ in the last entry.  Proposition \ref{liu31} implies that this is an irreducible subring matrix.

\item The entry $a_{i_{j_1} i_{j_2}}$ is contained in one of the first $k$ columns of the irreducible subring matrix described in Part (2).  Proposition \ref{liu31} implies that $p\mid a_{i_{j_1} i_{j_2}}$.

\item Suppose $i \in I$ and consider $v_i$.  Part (3) implies that any entry of $v_i$ in a row corresponding to an element of $I$ is divisible by $p$.  Since $A$ is in Hermite normal form, the other entries of $v_i$ are $0$.
\end{enumerate}
\end{proof}	

	The next result is the key observation that allows us to use combinatorial properties of subring matrices to prove upper bounds for $h_{n,k}(p^e)$.
	\begin{thm}\label{corank_diag_thm}
		Let $A$ be a $n\times n$ subring matrix with diagonal $(p^{e_1}, \ldots, p^{e_{n-1}}, 1)$ and define $I = \{i_1,i_2,\ldots, i_k\}$ where $1 \le i_1<i_2<\cdots < i_k \le n-1$ and $e_j \neq 0$ if and only if $j \in I$. Then $\corank(A) = k$.
	\end{thm}
	
	The analogue of this statement for sublattices is not true. For example, the column span of the matrix in Hermite normal form 
 $$
A = \begin{pmatrix}
     2&1 & 1\\
     0 & 2 & 1\\
     0 & 0 &1
 \end{pmatrix}
 $$
 has corank 1 since the gcd of the $2 \times 2$ minors of $A$ is $1$, but $A$ has two diagonal entries that are positive powers of $2$.
	\begin{proof}
We bound the corank from above and below.  There is an $(n-k)\times (n-k)$ submatrix of $A$ that is upper triangular and has diagonal $(1,1,\ldots,1)$.  Therefore, Proposition \ref{SNF_prop} implies that $\corank(A) \le k$. 

Proposition \ref{prop:subring_matrix_divis} implies that if $i \in I$, then every entry of the column $v_i$ is divisible by $p$.  Therefore, every $(n-k+1)\times (n-k+1)$ submatrix of $A$ contains a column in which every entry is divisible by $p$, so every $(n-k+1)\times (n-k+1)$ minor of $A$ is divisible by $p$.  Proposition \ref{SNF_prop} then implies that $\corank(A) \ge k$, completing the proof.
	\end{proof}
 At the end of the introduction we discussed how cokernels of $n \times n$ subring matrices ordered by determinant are not distributed according to the Cohen-Lenstra heuristics.  For example, the cokernel of such a matrix is cyclic much less often than the heuristics would predict. We give a kind of rough explanation for why this should be true.  A main issue is that the number of subring matrices of index equal to a power of $p$ that have cyclic cokernel is quite small.  If we require that the column space of a matrix $A$ in Hermite normal form is closed under componentwise multiplication, a large collection of entries is forced to be congruent to $0$ modulo $p$, and so it seems much more likely that every $k\times k$ minor of $A$ is divisible by $p$.  Proposition \ref{SNF_prop} then explains why we should expect the corank of $A$ not to be too small.

	We now prove Proposition \ref{count_cocyclic}, which counts cocyclic subrings of $\Z^n$ of determinant $p^e$.
	\begin{proof}[Proof of Proposition \ref{count_cocyclic}]
		Suppose $A \in \HH_n(\Z)$ is a subring matrix with determinant $p^e$ and that $\col(A)$ is a cocyclic subring of $\Z^n$.  By Theorem \ref{corank_diag_thm} exactly one diagonal entry of $A$ is equal to $p^e$ and the rest of the diagonal entries are equal to $1$.  We first consider the case where $A$ has the form
		\[
		\begin{pmatrix}
			p^e & a_{12} & a_{13} & \cdots & a_{1n}\\
			&1&0&\cdots &0\\
			&&1&\cdots &0\\
			&&&\ddots &\vdots\\
			&&&&1
		\end{pmatrix}.
		\]
Since $\col(A)$ is multiplicatively closed, for any $i$ satisfying $2 \le i\le n$ there exist unique $c_1,c_2\ldots, c_n \in \Z$ such that
\[
v_i \circ v_i - v_i = \sum_{\ell=1}^n c_\ell v_\ell.
\]		
Since the only nonzero entry in $v_i \circ v_i - v_i$ is an $a_{1i}^2-a_{1i}$ in the first row, we must have $p^e \mid (a_{1i}^2-a_{1i})$.  If $i\neq j$ satisfy $2\le i,j \le n$, then the only nonzero entry of $v_i \circ v_j$ is an $a_{1i}a_{1j}$ in the first row.  So we must have $p^e \mid a_{1i}a_{1j}$.
These conditions are satisfied if and only if there is at most one $j$ for which $p^e \mid (a_{1j} -1)$, and $p^e \mid a_{1i}$ for all $i \ne j$. We cannot have every $a_{1i}$ divisible by $p$ since $\col(A)$ must contain $(1,1,\ldots, 1)^T$.  Thus, there are a total of $n-1$ cocyclic subring matrices with this diagonal.
		
		For each $i \in [2,n-1]$ we repeat this analysis for the case where $a_{ii} = p^e$.  In this case there are exactly $n-i$ cocyclic subring matrices.  Taking a sum over $i$ completes the proof.
\end{proof}

In order to prove Theorems \ref{corank2_upper} and \ref{corank3_upper}, we need the following propositions about the structure of subring matrices with fixed diagonal and fixed corank. The first states that for any $i,j$ such that $e_i \neq 0$ and $e_j = 0$, the entry $a_{ij} \in \{0,1\}$. The second proposition states that for a fixed $i$ satisfying $e_i \neq 0$, there is exactly one $j$ for which $a_{ij} =1$. 

\begin{prop}\label{entry0or1prop}
	Let $A$ be a subring matrix with diagonal $(p^{e_1}, \ldots, p^{e_{n-1}}, 1)$ and define $I = \{i_1,i_2,\ldots, i_k\}$ where $1 \le i_1<i_2<\cdots < i_k \le n-1$ and $e_j \neq 0$ if and only if $j \in I$. If $i \in I$ and $j \not\in I$, then $a_{ij} \in \{0,1\}$.
\end{prop}
\begin{proof}
Since $\col(A)$ is multiplicatively closed there exist unique $c_1,\ldots c_n \in \Z$ such that 
\[
v_j \circ v_j - v_j = \sum_{\ell=1}^n c_\ell v_\ell.
\]
All of the nonzero entries of $v_j \circ v_j - v_j$ are contained in rows $1,2,\ldots, j-1$.  Therefore, $c_\ell =0$ if $\ell \ge j$.

If $\ell \not\in I$, then $e_\ell = 0$ and $p^{e_\ell} = 1$.  Since $A$ is in Hermite normal form the only nonzero entry in the $\ell$\textsuperscript{th} row of $A$ is this $1$ on the diagonal.  Therefore, if $\ell < j$ satisfies $e_\ell = 0$, then $c_\ell = 0$.

Consider the subset of $I$ consisting of integers less than $j$.  That is, let $I' = \{i_1, i_2,\ldots, i_m\}$ be defined so that $i_m < j$ but $i_{m+1} > j$, or if there is no such $m$, then $i_k < j$ and we define $I' = I$.  

The entry of $v_j \circ v_j - v_j$ in row $i_m$ is $(a_{i_m j}^2 -a_{i_m j }) = a_{i_m j}(a_{i_m j}-1)$.  Note that $p$ cannot divide both $a_{i_m j} $ and $a_{i_m j }-1$.  The entry of $\sum_{\ell=1}^n c_\ell v_\ell = \sum_{\ell=1}^{i_m} c_\ell v_\ell$ in row $i_m$ is $c_{i_m} p^{e_{i_m}}$.  Since $A$ is in Hermite normal form, we have $0 \le a_{i_m j} < p^{e_{i_m}}$ and see that $p^{e_{i_m}} \mid a_{i_m j }(a_{i_m j}-1)$.  This implies $a_{i_m j} = 0$ or $a_{i_m j} = 1$.

We now see that it is not possible for $v_j \circ v_j - v_j$ to have any nonzero entries outside of rows $1,2,\ldots, i_{m-1}$. That is, $\sum_{\ell=1}^n c_\ell v_\ell = \sum_{\ell=1}^{i_{m-1}} c_\ell v_\ell$.  Applying the argument that we just gave to row $i_{m-1}$ shows that $a_{i_{m-1}j} \in \{0,1\}$.  This implies that it is not possible for $v_j \circ v_j - v_j$ to have any nonzero entries outside of rows $1,2,\ldots, i_{m-2}$.  Arguing by induction shows that for any $i \in I$ we have $a_{ij} \in \{0,1\}$.
\end{proof}

\begin{prop}\label{prop_exactly_1}
	Let $A$ be a subring matrix with diagonal $(p^{e_1}, \ldots, p^{e_{n-1}}, 1)$ and define $I = \{i_1,i_2,\ldots, i_k\}$ where $1 \le i_1<i_2<\cdots < i_k \le n-1$ and $e_j \neq 0$ if and only if $j \in I$.  
\begin{enumerate} 
\item Suppose $i \in I$ and $j_1, j_2 \in \{1,2,\ldots, n\} \setminus I$ with $j_1 \neq j_2$.  Either $a_{ij_1} = 0$ or $a_{ij_2} = 0$.	
\item Suppose $i\in I$.  There exists exactly one $j \in \{1,2,\ldots, n\} \setminus I$ for which $a_{ij} = 1$.
\end{enumerate}
\end{prop}
Before giving the proof, we note that one can think of this proposition as a more general form of a result from the proof of Proposition \ref{count_cocyclic}.

\begin{proof}
Since $\col(A)$ is multiplicatively closed there exist unique $c_1,\ldots, c_n$ such that 
$
v_{j_1} \circ v_{j_2} = \sum_{\ell = 1}^n c_{\ell} v_\ell.
$ 
We will prove that every entry of $v_{j_1} \circ v_{j_2}$ is $0$ by showing that $c_\ell = 0$ for each $\ell$.  This shows that in each row, at most one of these two columns contains a nonzero entry.

Since $A$ is in Hermite normal form, if $\ell \not\in I$, then the only nonzero entry in the $\ell$\textsuperscript{th} row of $A$ is a $1$ on the diagonal.  Since $j_1 \neq j_2$, the entry of $v_{j_1} \circ v_{j_2}$ in row $\ell$ is $0$.  Therefore, $c_\ell = 0$.  So we see that $v_{j_1} \circ v_{j_2} = \sum_{\ell = 1}^k c_{i_\ell} v_{i_\ell}$.

We first prove that $c_{i_k} = 0$.  The entry of $v_{j_1} \circ v_{j_2}$ in row $i_k$ is $a_{i_k j_1} a_{i_k j_2}$.  Proposition \ref{entry0or1prop} implies that if $a_{i_k j_1} \neq 0$, then $a_{i_k j_1} =1$, and the corresponding statement holds for $a_{i_k j_2}$.  Therefore, if both entries are nonzero, then $a_{i_k j_1} a_{i_k j_2} = 1$.  The entry of $\sum_{\ell = 1}^k c_{i_\ell} v_{i_\ell}$ in row $i_k$ is $c_{i_k} p^{e_{i_k}}$.  Since $p^{e_{i_k}}\nmid 1$, we must have $c_{i_k} = a_{i_k j_1} a_{i_k j_2} = 0$.

We next consider the entry of $v_{j_1} \circ v_{j_2}$ in row $i_{k-1}$.  Repeating the argument just given shows that $c_{i_{k-1}} = 0$ and at least one of $a_{i_{k-1} j_1}, a_{i_{k-1} j_2}$ is $0$.  Repeating this argument for the rows $i_{k-2}, i_{k-3}, \ldots, i_1$ in order we see that for any $i \in I$, either $a_{i j_1} = 0$ or $a_{i j_2} = 0$.

Part (1) of this proposition implies that there is at most one $j \in \{1,2,\ldots, n\} \setminus I$ for which $a_{ij} \neq 0$.  Proposition \ref{entry0or1prop} implies that if $a_{ij} \neq 0$ then $a_{ij} = 1$.  Proposition \ref{prop:subring_matrix_divis} (3) implies that if for every $j \in \{1,2,\ldots, n\} \setminus I$ we have $a_{ij} = 0$, then every entry of the $i$\textsuperscript{th} row of $A$ is divisible by $p$.  This is not possible since $\col(A)$ must contain $(1,1,\ldots, 1)^T$.
\end{proof}

	Before applying Theorem \ref{corank_diag_thm} to prove Theorems \ref{corank2_upper}, \ref{corank3_upper}, and \ref{general_upper}, we recall one more result of Liu on the structure of subring matrices.
	\begin{prop}\cite[Lemma 3.5]{Liu}\label{liulastcol}
		If $A$ is a subring matrix, then every entry in the $n$\textsuperscript{th} column of $A$ is in $\{0,1\}$.  If $a_{in} = 1$ and $a_{jn} =0$, then $a_{ij} =a_{ji}=0$.
	\end{prop}

\begin{proof}[Proof of Theorem \ref{corank2_upper}]
Suppose $A$ is an $n \times n$ subring matrix with diagonal $(p^{e_1},\ldots, p^{e_{n-1}},1)$ and $\col(A)$ has corank $2$.  Theorem \ref{corank_diag_thm} implies that there exist $i,j$ satisfying $1\le i < j \le n-1$ where $e_i,e_j \ge 1$ and $e_\ell = 0$ for all $\ell \in [1,n-1]\setminus \{i,j\}$.  Let $\alpha = (e_i,e_j)$.  An example of such a matrix with $i = 1$ is
	\[
	A = \begin{pmatrix}
		p^{e_1} & a_{12} &a_{13}& \ldots &\ldots&\ldots& a_{1n}\\
		& 1 & 0 & \ldots &\ldots&\ldots& 0\\
		&&\ddots&&&\\
		&&&p^{e_j}&a_{j(j+1)} &\ldots& a_{jn}\\
		&&&&1&\ldots&0\\
		&&&&&\ddots&\\
		&&&&&&1
	\end{pmatrix}.
	\]

Since $A$ is in Hermite normal form, if $\ell \in [1,n-1]\setminus \{i,j\}$ then $p^{e_\ell} = 1$ and the only nonzero entry in the $\ell$\textsuperscript{th} row of $A$ is this $1$ on the diagonal.  Proposition \ref{prop_exactly_1} implies that there exists a unique $\ell_i \in [1,n] \setminus \{i,j\}$ where $a_{i \ell_i} = 1$, and for every $\ell \in [1,n] \setminus \{i,j,\ell_i\}$ we have $a_{i \ell} = 0$.  The corresponding statement holds for the $j$\textsuperscript{th} row of $A$.  There is a unique $\ell_j \in [1,n] \setminus \{i,j\}$ where $a_{j \ell_j} = 1$ and for every $\ell \in [1,n] \setminus \{j,\ell_j\}$ we have $a_{j \ell} = 0$.

We count the number of subring matrices $A$ where the diagonal entry in row $i$ is $p^{e_i}$, the diagonal entry in row $j$ is $p^{e_j}$, and all other diagonal entries are equal to $1$.  We will prove that there are 
\begin{equation}\label{ei_ej_count}
(n-j-1) + (n-i-2) + g_\alpha(p) + g_\alpha(p) (n-i-2)(n-j-1)
\end{equation}
such matrices by dividing them up based on the pair $(a_{in},a_{jn})$.  Proposition \ref{entry0or1prop} implies that $(a_{in},a_{jn}) \in \{(1,0),(0,1),(1,1),(0,0)\}$.  We consider each of these possibilities separately.
\begin{enumerate}
\item $(a_{in},a_{jn}) = (1,0)$.

Proposition \ref{liulastcol} implies that $a_{ij} = 0$.  There are $n-j-1$ choices for which entry in row $j$ is equal to $1$.  Any such choice completely determines the entries of $A$.  It is easy to check that each of these choices does give a subring matrix.

\item $(a_{in},a_{jn}) = (0,1)$.

Proposition \ref{liulastcol} implies that $a_{ij} = 0$.  There are $n-i-2$ choices for which entry in row $i$ is equal to $1$.  Any such choice completely determines the entries of $A$.  It is easy to check that each of these choices does give a subring matrix.

\item $(a_{in},a_{jn}) = (1,1)$.

Proposition \ref{prop:subring_matrix_divis} (2) implies that if $A$ is a subring matrix, then
	\[
	\begin{pmatrix}
		p^{e_i} & a_{ij} & 1\\
		0 & p^{e_j} & 1 \\ 
		0 & 0 & 1
	\end{pmatrix}
	\]
is an irreducible subring matrix.  This matrix is completely determined by $a_{ij}$.  There are $g_\alpha(p)$ choices for this entry.  Any such choice completely determines the entries of $A$.  It is easy to check that each of these choices does give a subring matrix.

\item $(a_{in},a_{jn}) = (0,0)$.

As in the previous case, there are $g_\alpha(p)$ choices for the entry $a_{ij}$.  There are $n-i-2$ choices for which entry in row $i$ is equal to $1$ and $n-j-1$ choices for which entry in row $j$ is equal to $1$.  These three choices are independent and completely determine the entries of $A$.  It is easy to check that every set of choices gives a subring matrix.

\end{enumerate}
It is straightforward to prove by induction that 
\[
\sum_{i=1}^{n-2} \sum_{j=i+1}^{n-1}  \left((n-i-2)+(n-j-1)\right) = 3 \binom{n-1}{3}
\]
and that 
\[
\sum_{i=1}^{n-2} \sum_{j=i+1}^{n-1}  \left(1+(n-i-2)(n-j-1)\right) = \frac{3n^2-17n+36}{12} \binom{n-1}{2}.
\]
There are $e-1$ pairs $(e_i,e_j)$ of positive integers with $e_i+e_j = e$.  The sum of $g_\alpha(p)$ taken over all compositions $\alpha$ of $e$ into two parts is $g_3(p^e)$.  Adding these terms together gives the formula in Theorem \ref{corank2_upper}.
\end{proof}

\begin{proof}[Proof of Theorem \ref{corank3_upper}]
We closely follow the strategy of the proof of Theorem \ref{corank2_upper}.  Suppose $A$ is an $n \times n$ subring matrix with diagonal $(p^{e_1},p^{e_2},\ldots, p^{e_{n-1}},1)$ and $\col(A)$ has corank $3$.  Theorem \ref{corank_diag_thm} implies that there exist $i,j,k$ satisfying $1\le i < j < k \le n-1$ where $e_i,e_j, e_k \ge 1$ and $e_\ell = 0$ for all $\ell \in [1,n-1]\setminus \{i,j,k\}$.  If $\ell \in [1,n-1]\setminus \{i,j,k\}$, then the only nonzero entry in the $\ell$\textsuperscript{th} row of $A$ is a $1$ on the diagonal.  Proposition \ref{prop_exactly_1} implies that there exists a unique $\ell_i \in [1,n]\setminus \{i,j,k\}$ where $a_{i \ell_i} = 1$, and for every $\ell\in [1,n] \setminus \{i,j,k,\ell_i\}$ we have $a_{i \ell} = 0$.  Similarly, there exists a unique $\ell_j \in [1,n] \setminus \{j,k\}$ where $a_{j \ell_j} = 1$ and for every $\ell \in [1,n] \setminus \{j,k,\ell_j\}$ we have $a_{j \ell} = 0$.  Also, there exists a unique $\ell_k \in [1,n] \setminus \{k\}$ where $a_{k \ell_k} = 1$, and for every $\ell \in [1,n] \setminus \{k,\ell_k\}$ we have $a_{k \ell} = 0$.

Proposition \ref{entry0or1prop} implies that $a_{in},a_{jn}, a_{kn} \in \{0,1\}$.  We count subring matrices $A$ where the diagonal entry in row $i$ is $p^{e_i}$, the diagonal entry in row $j$ is $p^{e_j}$, and the diagonal entry in row $k$ is $p^{e_k}$, by dividing up these matrices based on the $8$ possibilities for $(a_{in},a_{jn}, a_{kn})$.

Let $\alpha = (e_i,e_j,e_k),\ \alpha_{12} = (e_i,e_j),\ \alpha_{13} = (e_i, e_k)$, and $\alpha_{23} = (e_j,e_k)$.  We prove that the number of subring matrices with a particular value of $(a_{in}, a_{jn}, a_{kn})$ is given in the following table:
\[
\begin{tabular}{|c|c|p{8.0cm}|}
\hline
{$(a_{in}, a_{jn}, a_{kn})$} & $\#\{\text{Subring Matrices } A\}$  \\[2pt]
\hline
$(1,1,1)$ & $g_\alpha(p)$\\[2pt]
\hline
$(0,0,0)$ & $(n-i-3)(n-j-2)(n-k-1) g_\alpha(p)$ \\[2pt]
\hline
$(1,1,0)$ & $(n-k-1) g_{\alpha_{12}}(p)$ \\[2pt]
\hline
$(1,0,1)$ & $(n-j-2) g_{\alpha_{13}}(p)$ \\[2pt]
\hline
$(0,1,1)$ & $(n-i-3) g_{\alpha_{23}}(p)$ \\[2pt]
\hline
$(1,0,0)$ & $(n-j-2)(n-k-1) g_{\alpha_{23}}(p)$ \\[2pt]
\hline
$(0,1,0)$ & $(n-i-3)(n-k-1) g_{\alpha_{13}}(p)$ \\[2pt]
\hline
$(0,0,1)$ & $(n-i-3)(n-j-2) g_{\alpha_{12}}(p)$ \\[2pt]
\hline
\end{tabular}.
\]

Assuming for now that the values in this table are correct we complete the proof of the theorem.  It is straightforward to prove the following formulas by induction:
\begin{eqnarray*}
\sum_{i=1}^{n-3} \sum_{j=i+1}^{n-2} \sum_{k=j+1}^{n-1} \left( (n-i-3)(n-j-2) + (n-k-1) \right) & = & \frac{1}{5} \binom{n-1}{4} (8n-25), \\
\sum_{i=1}^{n-3} \sum_{j=i+1}^{n-2} \sum_{k=j+1}^{n-1} \left( (n-i-3)(n-k-1) + (n-j-2) \right) & = & \frac{1}{5} \binom{n-1}{4} (4n-5), \\\
\sum_{i=1}^{n-3} \sum_{j=i+1}^{n-2} \sum_{k=j+1}^{n-1} \left( (n-j-2)(n-k-1) + (n-i-3) \right) & = & \frac{1}{5}\binom{n-1}{4} (3n+5),
\end{eqnarray*}
and
\begin{eqnarray*}
\sum_{i=1}^{n-3} \sum_{j=i+1}^{n-2} \sum_{k=j+1}^{n-1} \left((n-i-3)(n-j-2)(n-k-1) +1\right) & = & \frac{1}{8} \binom{n-1}{3} (n^3-11n^2+40n-40).
\end{eqnarray*}

There is a bijection between compositions $\alpha = (e_i,e_j,e_k)$ of $e$ into three parts and compositions $(e_i,e_j)$ of an integer $m \in [2,e-1]$ into two parts.  The number of compositions of $m$ into two parts is $m-1$.  Taking a sum over all possible compositions $\alpha$ gives
\[
\sum_\alpha g_{\alpha_{12}}(p) = \sum_{m=2}^{e-1} (m-1) g_3(p^m).
\]
For exactly the same reason we get the same expression when we sum $g_{\alpha_{13}}(p)$ or $g_{\alpha_{23}}(p)$ over this set of  $\alpha$.  The sum of $g_\alpha(p)$ taken over all compositions of $e$ into three parts is $g_4(p^e)$.  Combining these facts with the observation that $(8n-25)+(4n-5)+(3n+5) = 5(3n-5)$ completes the proof.

We now prove that the values in the table given above are correct.
\begin{enumerate}
\item $(a_{in},a_{jn},a_{kn}) = (1,1,1)$.

Proposition \ref{prop:subring_matrix_divis} (2) implies that if $A$ is a subring matrix, then
	\[
	\begin{pmatrix}
		p^{e_i} & a_{ij} & a_{ik} & 1\\
		0 & p^{e_j} & a_{jk} & 1 \\ 
		0 & 0 & p^{e_k}& 1 \\ 
		
		0 & 0 & 0& 1
	\end{pmatrix}
	\]
is an irreducible subring matrix.  There are $g_{\alpha}(p)$ possibilities for the triple of entries $a_{ij},a_{ik},a_{jk}$.  Any such choice completely determines the entries of $A$.  It is easy to check that each of these choices does give a subring matrix.

\item $(a_{in},a_{jn},a_{kn}) = (0,0,0)$.

As in the previous case, there are $g_\alpha(p)$ choices for the triple of entries $a_{ij},a_{ik},a_{jk}$.  There are $n-i-3$ choices for which entry in row $i$ is equal to $1$, $n-j-2$ choices for which entry in row $j$ is equal to $1$, and $n-k-1$ choices for which entry in row $k$ is equal to $1$.  These choices are independent and completely determine the entries of $A$.  It is easy to check that every set of choices gives a subring matrix.

\end{enumerate}

We discuss two of the six additional possibilities for $(a_{in},a_{jn},a_{kn})$.  The others cases are very similar.
\begin{enumerate}\setcounter{enumi}{2}
\item $(a_{in},a_{jn},a_{kn}) = (0,1,1)$.

Proposition \ref{liulastcol} implies that $a_{ij} = a_{ik} =  0$.  Proposition \ref{prop:subring_matrix_divis} (2) implies that if $A$ is a subring matrix, then
	\[
	\begin{pmatrix}
		p^{e_i} & 0 & 0 & 1\\
		0 & p^{e_j} & a_{jk} & 1 \\ 
		0 & 0 & p^{e_k}& 1 \\ 
		
		0 & 0 & 0& 1
	\end{pmatrix}
	\]
is an irreducible subring matrix. This is an irreducible subring matrix if and only if 
	\[
	\begin{pmatrix}
		p^{e_j} & a_{jk} & 1 \\ 
		0 & p^{e_k}& 1 \\ 
		0 & 0& 1
	\end{pmatrix}
	\]
is an irreducible subring matrix.  There are $g_{\alpha_{23}}(p)$ such matrices.  There are $n-i-3$ choices for which entry in row $i$ is equal to $1$.  These choices are independent and completely determine the entries of $A$.  It is easy to check that every set of choices gives a subring matrix.

\item $(a_{in},a_{jn},a_{kn}) = (0,1,0)$.

Proposition \ref{liulastcol} implies that $a_{ij} = a_{jk} =  0$.  Proposition \ref{prop:subring_matrix_divis} (2) implies that if $A$ is a subring matrix, then
	\[
	\begin{pmatrix}
		p^{e_i} & 0 & a_{ik} & 1\\
		0 & p^{e_j} &  0  & 1 \\ 
		0 & 0 & p^{e_k}& 1 \\ 
		0 & 0 & 0& 1
	\end{pmatrix}
	\]
is an irreducible subring matrix. This is an irreducible subring matrix if and only if 
	\[
	\begin{pmatrix}
		p^{e_i} & a_{ik} & 1 \\ 
		0 & p^{e_k}& 1 \\ 
		0 & 0& 1
	\end{pmatrix}
	\]
is an irreducible subring matrix.  There are $g_{\alpha_{13}}(p)$ such matrices.  There are $n-i-3$ choices for which entry in row $i$ is equal to $1$ and $n-k-1$ choices for which entry in row $k$ is equal to $1$.  These choices are independent and completely determine the entries of $A$.  It is easy to check that every set of choices gives a subring matrix.

\end{enumerate}
We omit the four remaining cases since they are very similar to the ones described here.

\end{proof}

We now give the proof of Theorem \ref{general_upper}.
\begin{proof}[Proof of Theorem \ref{general_upper}]
Theorem \ref{corank_diag_thm} implies that a subring of $\Z^n$ of corank $k$ and index $p^e$ is the column span of an $n \times n$ matrix $A$ with diagonal $(p^{e_1}, p^{e_2},\ldots, p^{e_{n-1}},1)$ where exactly $k$ of these first $n-1$ diagonal entries are positive powers of $p$.  Suppose these entries are $p^{i_1}, p^{i_2}, \ldots, p^{i_k}$ where $1 \le i_1 < i_2< \cdots < i_k \le n-1$. Let $\alpha = (p^{e_{i_1}},p^{e_{i_2}},\ldots, p^{e_{i_k}})$. Note that $\alpha$ is a composition of $e$ into $k$ parts.

Proposition \ref{prop:subring_matrix_divis} (2) implies that
\[
	\begin{pmatrix}
		p^{e_{i_1}} & a_{i_1 i_2} & \cdots & a_{i_1 i_k} & 1\\
		0 & p^{e_{i_2}} &  \cdots & a_{i_2 i_k}   & 1 \\ 
		0 & 0 & \ddots & \vdots & \vdots \\
		0 & 0 & 0 & p^{e_k}& 1 \\ 
		0 & 0 & 0& 0 &  1
	\end{pmatrix}
\]
is an irreducible subring matrix.  There are $g_\alpha(p)$ choices for the entries of this matrix.  Proposition \ref{prop_exactly_1} (2) implies that in each of row $i_1, i_2,\ldots, i_k$ of the matrix $A$, there is precisely one entry equal to $1$.  Once we make a choice for where these entries are, we have completely determined the entries of the matrix $A$.   

For the lower bound, we note that if we choose each of these entries to be in the final column of $A$, then it is easy to check that we do get a subring matrix.  For the upper bound, we note that in each one of these rows, the entry equal to $1$ cannot lie in the columns $i_1, \ldots, i_k$, so there are at most $n-k$ choices for where this $1$ could be.  This means that the number of choices for $A$ where the collection of rows in which the diagonal entry is not equal to $1$ is $i_1, i_2, \ldots, i_k$ and these diagonal entries give the composition $\alpha$ is at least $g_\alpha(p)$ and is at most $(n-k)^k g_\alpha(p)$.  Taking a sum over all $\binom{n-1}{k}$ choices for $(i_1,\ldots, i_k)$ and over all compositions $\alpha$ of $e$ into $k$ parts completes the proof of the first part of the theorem.

For the second part, recall that $\tilde{h}_{n,k}(p^e) = \sum_{\ell=1}^{k} h_{n,\ell}(p^e)$.  For any positive integer $\ell$, there is a trivial upper bound $\binom{n-1}{\ell} \le 2^{n-1}$.  It is easy to see that for any $\ell \le k$, we have $g_{\ell+1}(p^e) \le f_{\ell+1}(p^e) \le f_{k+1}(p^e)$.  We conclude that $\tilde{h}_{n,k}(p^e)$ is at most 
\[   
\sum_{\ell=1}^{k} (n-\ell)^\ell \binom{n-1}{\ell} g_{\ell+1}(p^e) 
 \le  \sum_{\ell=1}^{k} (n-1)^k 2^{n-1} f_{k+1}(p^e) 
 \le  k (n-1)^k 2^{n-1} f_{k+1}(p^e).
\]
\end{proof}
 
\section{Subrings of $\Z^n$ of given corank for $n \le 5$}\label{corank_smalln}

When $n \ge 6$, Theorem \ref{proportion_corank_k} states that the proportion of subrings in $\Z^n$ with `small' corank is 0\%. In this section, we study subrings of $\Z^n$ with given corank when $n \le 5$. For each $n \le 4$ we compute the proportion of subrings of each fixed corank.  For $n =5$ we show that the proportion of subrings of each fixed corank is positive, but we cannot determine exactly what these proportions are.  Theorem \ref{proportion_corank_k} (1) implies that 100\% of subrings of $\Z^6$ have corank 4 or 5. We are unable to determine whether a positive proportion of these subrings have corank $4$ because we cannot currently determine the order of the pole at $s=1$ of $H_{6,4}(X)$. We will see below that a positive proportion of these subrings have corank $5$.

In Proposition \ref{subringZ2} we saw that for each positive integer $k$ there is a unique subring of $\Z^2$ of index $k$ and that every proper subring of $\Z^2$ has corank $1$.  We next consider subrings of~$\Z^3$.

\begin{cor}
	We have 
	\begin{align*}
	p_{3,1}^R &= \zeta(2)\prod_p (p^{-3}(p-1)^2(p+2)) \approx .471683\\
 p_{3,2}^R &= 1 - p_{3,1}^R \approx .528317.
	\end{align*}
\end{cor}

\begin{proof}
	This follows from Corollary \ref{cor:cocyclic_asymptotic}, Corollary \ref{Nn_growth}, and the fact that every proper subring of $\Z^3$ has corank $1$ or $2$.
\end{proof}

We next consider subrings of $\Z^4$. Every proper subring in $\Z^4$ has corank $1,2$, or $3$. Specializing the formula for $\zeta_{\Z^n}^{R,(2)}(s)$ given in Proposition \ref{prop:corank2_zeta} to the case $n=4$ gives the following.

\begin{cor}\label{crk2_4_cor}
We have 
	\begin{align*}
		\zeta_{\Z^4}^{R,(2)}(s) &=  \zeta(3s-1)\zeta(s)^6 \prod_p\bigg( 1 - 6 p^{1 - 9 s} + 24 p^{1 - 8 s} - 33 p^{1 - 7 s}\\
	&+	12 p^{1 - 6 s} + 12 p^{1 - 5 s} - 12 p^{1 - 4 s} + 3 p^{1 - 3 s}- 
		4 p^{-7 s} + 18 p^{-6 s}\\
		&- 28 p^{-5 s}+ 13 p^{-4 s }+ 8 p^{-3s} - 
		8 p^{-2s}\bigg).
	\end{align*}
\end{cor}

This expression leads directly to the following result.
\begin{cor}
	We have
	\begin{align*}
		p_{4,1}^R &=\zeta(2)^3\prod_p\left( p^{-6}(p - 1)^5(p + 5)\right)\approx .0593079 \\
 p_{4,2}^R &= 	\zeta(2)^4\prod_p p^{-8}\left((p-1)^5 (1+p) (p^2+4p+6)\right) - p_{4,1}^R \approx .4389531\\
		p_{4,3}^R &= 1 - p_{4,1}^R - p_{4,2}^R \approx .501739.
	\end{align*}
\end{cor}
\begin{proof}
	This follows from Corollaries \ref{Nn_growth}, \ref{cor:cocyclic_asymptotic},  \ref{crk2_4_cor},  and from the fact that every proper subring of $\Z^4$ has corank $1,2$, or $3$.
\end{proof}

Theorem \ref{N5X} says that there exists a positive real number $C_5$ such that 
\[
N_5(X) \sim C_5 X (\log X)^9.
\]
However, it is not currently known what this constant $C_5$ is.  Theorem \ref{Hnk_123} says that $H_{5,1}(X), H_{5,2}(X)$, and $H_{5,3}(X)$ each have similar rates of growth.  We can conclude that $p_{5,k}^R >0$ for each $k \in \{1,2,3\}$, but we are not able to determine the exact values of these probabilities.  We can determine their relative frequencies, for example, $\frac{p_{5,2}^R}{p_{5,1}^R} \approx 59.801$ and $\frac{p_{5,3}^R}{p_{5,1}^R} \approx 679.548$.  We now prove that $p_{5,4}^R > 0$, or equivalently, that $p_{5,1}^R + p_{5,2}^R + p_{5,3}^R < 1$.  We do this by proving more generally that $p_{n,n-1}^R>0$ for any positive integer $n$.

For the rest of the paper, if $G$ is a finite abelian group we write $G_p$ for its Sylow $p$-subgroup.  
\begin{thm}\label{corank_n-1}
Let $n$ be a positive integer and $p$ be a prime.  Then
\[
\lim_{X\rightarrow \infty} \frac{\#\{\text{Subrings } R \subseteq \Z^n\colon [\Z^n\colon R] \le X \text{ and } (\Z^n/R)_p \cong (\Z/p\Z)^{n-1}\} }{N_n(X)} > 0.
\]
Therefore, $p^R_{n,n-1} > 0$.
\end{thm}
The second statement follows from the first by noting that a finite abelian group has rank at most $r$ if and only if each of its Sylow $p$-subgroups has rank at most $r$.  If $R \subseteq \Z^n$ is a finite index subring then $\Z^n/R$ is a finite abelian group of rank at most $n-1$.  Therefore, if $(\Z^n/R)_p$ has rank $n-1$, then $\Z^n/R$ has rank $n-1$.

The main idea in the proof of Theorem \ref{corank_n-1} is to break $R$ up into `prime power components'.  Before explaining what we mean exactly, we highlight a particular subring of $\Z^n$ that plays an important role in our argument.
\begin{prop}\label{unique_Rpstar}
Let $n \ge 2$ be a positive integer, $p$ be a prime, and $e_1,e_2,\ldots, e_n$ be the standard basis vectors of $\R^n$.  There is a unique subring $R_p^* \subsetneq \Z^n$ such that $\Z^n/R_p^* \cong (\Z/p\Z)^{n-1}$.  This subring is generated by $pe_1, pe_2,\ldots, pe_{n-1}, e_1+\cdots + e_n$.
\end{prop}

\begin{proof}
Suppose $R \subseteq \Z^n$ is a subring for which $\Z^n/R \cong (\Z/p\Z)^{n-1}$.  It is given by $\col(A)$ for an $n \times n$ subring matrix $A$.  Since $[\Z^n\colon \col(A)] = p^{n-1}$ and $\col(A)$ has corank $n-1$, Theorem \ref{corank_diag_thm} implies that the diagonal of $A$ must be $(p,p,\ldots, p, 1)$, and Proposition \ref{liu33} implies that $A$ is an irreducible subring matrix.  Proposition \ref{liu31} implies that every entry in the last column of $A$ is $1$ and all of the other nonzero entries of $A$ are the entries equal to $p$ on the diagonal.  In particular, $A$ is unique.
\end{proof}

We need one additional piece in order to prove Theorem \ref{corank_n-1}.
\begin{prop}\label{p-part_trivial}
Let $n$ be a positive integer and $p$ be a prime.  Then
\[
\lim_{X\rightarrow \infty} \frac{\#\{\text{Subrings } R \subseteq \Z^n\colon [\Z^n\colon R] \le X \text{ and } p \nmid [\Z^n \colon R]\} }{N_n(X)}  = \zeta_{\Z^n,p}^R(\alpha)^{-1} > 0,
\]
where $\alpha$ is the abscissa of convergence of $\zeta_{\Z^n}^R(s)$.
\end{prop}

\begin{proof}
The asymptotic formula for the expression in the numerator comes from applying Theorem \ref{tauberian_theorem} to the function $\zeta_{\Z^n,p}^R(s)^{-1} \zeta_{\Z^n}^R(s)$.  Suppose that the rightmost pole of $\zeta_{\Z^n}^R(s)$ occurs at $s=\alpha$.  The results of \cite[Section 4]{dusautoy_grunewald} imply that the abscissa of convergence of each local factor $\zeta_{\Z^n,p}^R(s)$ occurs to the left of $\alpha$, and so $\zeta_{\Z^n,p}^R(\alpha)$ is a positive real number.  
The only difference in the Euler products defining the counting functions in the numerator and in the denominator of this proposition is the factor of $\zeta_{\Z^n,p}^R(s)$ that occurs in the denominator but not in the numerator.  Theorem \ref{tauberian_theorem} implies that the ratio in the proposition converges to $ \zeta_{\Z^n,p}^R(\alpha)^{-1}$.

\end{proof}

\begin{proof}[Proof of Theorem \ref{corank_n-1}]
Let $p_1,\ldots, p_r$ be distinct primes, $a_1, \ldots, a_r$ be positive integers, and $k = p_1^{a_1}\cdots p_r^{a_r}$.  The Euler product for $\zeta_{\Z^n}^R(s)$ reflects the fact that there is a bijection between subrings $R \subseteq \Z^n$ with $[\Z^n \colon R] = k$ and collections of subrings $(R_1,R_2,\ldots, R_r)$ where each $R_i \subseteq \Z^n$ is a subring with $[\Z^n\colon R_i] = p_i^{a_i}$.  For a description of how to find these subrings $R_i$ given the matrix $H$ in Hermite normal form for which $\col(H) = R$, see \cite[Section 4]{CKK}.  One can interpret this fact by noting that for any prime $p,\ \Z^n \hookrightarrow \Z_p^n$, so a subring $R \subseteq \Z^n$ gives a subring of $\Z_p^n$ for each $p$, where we get a proper subring if and only if $p \mid [\Z^n \colon R]$.

Proposition \ref{unique_Rpstar} implies that $(\Z^n/R)_p \cong (\Z/p\Z)^{n-1}$ if and only if $R_p = R_p^*$.  In this way, we see that there is a bijection between
\[
\mathcal{A}_X = \{\text{Subrings } R \subset \Z^n\colon [\Z^n \colon R] \le X \text{ and } (\Z^n/R)_p \cong (\Z/p\Z)^{n-1}\}
\]
and
\[
\mathcal{B}_X = \left\{\text{Subrings } R \subset \Z^n\colon [\Z^n \colon R] \le \frac{X}{p^{n-1}} \text{ and } p \nmid [\Z^n \colon R]\right\}.
\]
Proposition \ref{p-part_trivial} implies that as $X \rightarrow \infty$ the set $\mathcal{B}_X$ includes a positive proportion of all subrings of $\Z^n$ of index at most $X$.
\end{proof}


\appendix

\section*{Appendix: A conjecture for the cotype zeta function of $\Z^4$\\ (by Gautam Chinta)}

In this appendix we present a conjecture for the cotype zeta function of the ring $\Z^4$ and show how it is compatible with the results of Section \ref{sec_corank} and Section \ref{corank_smalln} on counting subrings of corank less than or equal to 3.  We begin with a definition of the cotype zeta function, which generalizes both the subring zeta function of Definition \ref{defn:subring_zeta} and the corank at most $k$ zeta function of Definition \ref{defn:corank_at_most}.  A finite index subring $R$ of $\Z^n$ will have corank at most $n-1$ since $(1,1,\ldots,1)\in R.$  Let $\alpha_1(R), \ldots, \alpha_{n-1}(R)$ be the invariant factors of the group $\Z^n/R$, where we set $\alpha_i(R)=1$ if $i$ is bigger than the corank of $R$.  For a tuple $\alpha=(\alpha_1,\ldots,\alpha_{n-1})$ of positive integers with $\alpha_{i+1}\mid\alpha_i$, set $f_{n}(\alpha)=f_n(\alpha_1,\ldots, \alpha_{n-1})$ to be the number of subrings of $\Z^n$ of cotype $\alpha.$

\begin{defnA}
  The {\em subring cotype zeta function of} $\Z^n$ is
\[\zeta_{\Z^n}^R(s_1, \ldots, s_{n-1})
=\sum_{\alpha_{n-1}\mid\alpha_{n-2}\mid\cdots\mid\alpha_1}
\frac{f_n(\alpha_1, \ldots, \alpha_{n-1})}{\alpha_1^{s_1}\cdots\alpha_{n-1}^{s_{n-1}}}.
\]
\end{defnA}
We will also refer to $\zeta_{\Z^n}^R(s_1, \ldots, s_{n-1})$ more simply as the {\em cotype zeta function of } $\Z^n$. Since $\alpha_1(R)\cdots \alpha_{n-1}(R)=[\Z^n:R]$ we have the relation
\begin{equation}
  \label{eq:app1}
  \zeta_{\Z^n}^R(s)=\zeta_{\Z^n}^R(s, \ldots, s).
\end{equation}
(We  hope no notational confusion will arise from letting the number of arguments distinguish between the single-variable subring zeta function on the left of (\ref{eq:app1}) and the multivariate cotype subring zeta function on the right.)  Just as with the single-variable subring zeta function, the cotype zeta function has an Euler product:
\begin{equation*}
  \zeta_{\Z^n}^R(s_1, \ldots, s_{n-1})=
\prod_p F_n(p;p^{-s_1},\dots, p^{-s_{n-1}})
\end{equation*}
for a rational function $F_n$ in $p^{-s_1},\dots, p^{-s_{n-1}}$.

A straightforward calculation yields
\begin{itemize}
\item For $\Z^2$:
$$F_{2}(p;x)=\frac 1{1-x}
$$
\item  For $\Z^3$:
$$F_3(p;x,y)=\frac {1+2x-2x^2y-x^3y}{(1-x)(1-xy)(1-px^2y)}
$$
\end{itemize}
where we have set $x=p^{-s_1}, y=p^{-s_2}.$

Note the functional equations
\begin{align} \label{eq:fes}
  F_2(1/p;1/x)&=-x\,F_2(p;x)\\ \nonumber
 F_3(1/p;1/x, 1/y)&=pxyF_3(p;x,y)
\end{align}
and the specialization
\begin{equation*}
F_{3}(p;x,x)=\frac{(1+x)^2}{(1-x)(1-px^3)},
\end{equation*}
in agreement with the local factor of $\zeta_{\Z^3}^R(s)$ in Theorem \ref{GeneratingFunctions}.

\subsection*{A conjecture for $\Z^4$}
The subring cotype zeta function of $\Z^n$ has not been explicitly computed for $n\geq 4$.  In this section, we give a conjecture for $n=4$ based on computer calculations.  By virtue of the Euler product, it suffices to define the local factor $F_4(p;x,y,z).$

\begin{conjA}\label{app:conjecture}
The local factor of the subring cotype zeta function of $\Z^4$ is
$$F_4(p;x,y,z)=\frac {N(p;x,y,z)}{D(p;x,y,z)}
$$
where the denominator is
\begin{equation*}
D(p;x,y,z)=
(1-x)(1-xy)(1-xyz)(1-px^2y)(1-p^2x^2yz)(1-p^2x^2y^2z)(1-p^3x^3y^2z)
\end{equation*}
and the numerator $N(p;x,y,z)$ is the polynomial of total degree 21 (in $x,y,z$) with coefficients as given in Table \ref{tab:1}.
\end{conjA}
\begin{table}[htbp]
 \parbox{.45\linewidth}{
\centering
  \begin{tabular}{|l|l|} \hline
monomial & coefficient \\ \hline \hline
$1$ & $1$ \\ \hline
$x$ & $5$ \\ \hline
$x y$ & $6$ \\ \hline
$x^{2} y$ & $3 \, p - 2$ \\ \hline
$x^{2} y z$ & $p - 5$ \\ \hline
$x^{3} y$ & $3 \, p - 4$ \\ \hline
$x^{2} y^{2} z$ & $p - 6$ \\ \hline
$x^{3} y z$ & $-p - 1$ \\ \hline
$x^{3} y^{2}$ & $-6 \, p$ \\ \hline
$x^{3} y^{2} z$ & $-5 \, p^{2} - 5 \, p + 1$ \\ \hline
$x^{4} y^{2}$ & $-6 \, p$ \\ \hline
$x^{3} y^{3} z$ & $-p - 1$ \\ \hline
$x^{4} y^{2} z$ & $-14 \, p^{2} - 3 \, p + 5$ \\ \hline
$x^{4} y^{3} z$ & $-4 \, p^{3} - 6 \, p^{2} + 6 \, p + 1$ \\ \hline
$x^{5} y^{2} z$ & $p^{2} + p$ \\ \hline
$x^{4} y^{3} z^{2}$ & $-p^{3} + 5 \, p^{2}$ \\ \hline
$x^{5} y^{3} z$ & $-7 \, p^{3} + 8 \, p^{2} + 8 \, p$ \\ \hline
$x^{5} y^{3} z^{2}$ & $-p^{4} + 13 \, p^{2}$ \\ \hline
$x^{5} y^{4} z$ & $p^{2} + p$ \\ \hline
$x^{6} y^{3} z$ & $5 \, p^{3} + 4 \, p^{2} - p$ \\ \hline
  \end{tabular}}
 \parbox{.45\linewidth}{
\centering
  \begin{tabular}{|l|l|}\hline
monomial & coefficient \\ \hline \hline
$x^{11} y^{7} z^{3}$ & $p^{5}$ \\ \hline
$x^{10} y^{7} z^{3}$ & $5 \, p^{5}$ \\ \hline
$x^{10} y^{6} z^{3}$ & $6 \, p^{5}$ \\ \hline
$x^{9} y^{6} z^{3}$ & $-2 \, p^{5} + 3 \, p^{4}$ \\ \hline
$x^{9} y^{6} z^{2}$ & $-5 \, p^{5} + p^{4}$ \\ \hline
$x^{8} y^{6} z^{3}$ & $-4 \, p^{5} + 3 \, p^{4}$ \\ \hline
$x^{9} y^{5} z^{2}$ & $-6 \, p^{5} + p^{4}$ \\ \hline
$x^{8} y^{6} z^{2}$ & $-p^{5} - p^{4}$ \\ \hline
$x^{8} y^{5} z^{3}$ & $-6 \, p^{4}$ \\ \hline
$x^{8} y^{5} z^{2}$ & $p^{5} - 5 \, p^{4} - 5 \, p^{3}$ \\ \hline
$x^{7} y^{5} z^{3}$ & $-6 \, p^{4}$ \\ \hline
$x^{8} y^{4} z^{2}$ & $-p^{5} - p^{4}$ \\ \hline
$x^{7} y^{5} z^{2}$ & $5 \, p^{5} - 3 \, p^{4} - 14 \, p^{3}$ \\ \hline
$x^{7} y^{4} z^{2}$ & $p^{5} + 6 \, p^{4} - 6 \, p^{3} - 4 \, p^{2}$ \\ \hline
$x^{6} y^{5} z^{2}$ & $p^{4} + p^{3}$ \\ \hline
$x^{7} y^{4} z$ & $5 \, p^{3} - p^{2}$ \\ \hline
$x^{6} y^{4} z^{2}$ & $8 \, p^{4} + 8 \, p^{3} - 7 \, p^{2}$ \\ \hline
$x^{6} y^{4} z$ & $13 \, p^{3} - p$ \\ \hline
$x^{6} y^{3} z^{2}$ & $p^{4} + p^{3}$ \\ \hline
$x^{5} y^{4} z^{2}$ & $-p^{4} + 4 \, p^{3} + 5 \, p^{2}$ \\ \hline
  \end{tabular}}
  \caption{Coefficients of the numerator $N(p;x,y,z)$}
  \label{tab:1}
\end{table}
This conjecture was obtained by first enumerating all the subrings (and cotypes) of $\Z_2^4$ of index less than $2^{23}$.  A similar computation was done for subrings of $\Z_p^4$ for $p=3,5$ and $7$.  Putting these computations together under the assumption that $N(p;x,y,z)$ is a polynomial in $p$ leads to the Conjecture.

\begin{rmkA}
  Note the functional equation
\[F(1/p;1/x,1/y,1/z)=-p^3xyzF(p;x,y,z).\]
 This functional equation and the ones in (\ref{eq:fes}) above are consistent with the results of Voll \cite{Voll2010}.
Furthermore, the specialization $F(p;p^{-s}, p^{-s}, p^{-s})$ agrees with the  local factor of $\zeta_{\Z^4}^R(s)$ in Theorem \ref{GeneratingFunctions}.
\end{rmkA}

\begin{rmkA}
We can show that the conjecture is compatible with the results of Section \ref{sec_corank} for $n=4$.  Explicitly, the corank at most $k$ zeta function of $\Z^n$ can be obtained from the cotype zeta function as follows:
\begin{equation}
  \label{eq:app5}
\zeta_{\Z^n}^{R,(k)}(s) = \lim_{s_{k+1}, \ldots, s_{n-1}\to\infty}
\zeta_{\Z^n}^{R}(s,\ldots,s,s_{k+1}, \ldots, s_{n-1}).
\end{equation}
When $n=4$ and $k=1$ or 2, we substitute the expression from Conjecture A\ref{app:conjecture} into the righthand side of (\ref{eq:app5}) to get
\begin{align*}
  \zeta_{\Z^4}^{R,(1)}(s) &= \prod_p\frac{ 1+5p^{-s}}{1-p^{-s}},\\
\zeta_{\Z^4}^{R,(2)}(s) &=\zeta(s)^6\cdot\prod_p\frac{\left( 1+4p^{-s}+2p^{-2s}+(3p-4)p^{-3s}-6p^{1-5s}\right)(1-p^{-s})^4}
{1-p^{1-3s}}.
\end{align*}
These expressions for the corank at most 1 and 2 zeta functions agree with Proposition \ref{cocyclic_zeta} when $n=4$ and Corollary \ref{crk2_4_cor}.
\end{rmkA}

\begin{rmkA}
Using the expression for the cotype zeta function of $\Z^4$, it should be possible to similarly obtain results for counting subrings of $\Z^n$ of corank 1,2 or 3 for arbitrary $n$.  This will follow from an analogue of Liu's recursion relation \cite[Proposition 4.4]{Liu} expressing reducible subrings in $\Z^n$ in terms subrings of $\Z^m$, for $m$ less than $n$.  In fact, suitable refinement of Liu's relation will identify subrings of $\Z^n$ of corank $m$ with products of subrings of $\Z^{m'}$ with $m'\leq m.$
\end{rmkA}

\subsection*{Acknowledgments}
The authors thank Gautam Chinta and Ramin Takloo-Bighash for helpful discussions, and in particular for letting us know about the lower bound for subrings of $\Z^6$ given in Proposition \ref{Z6_lower}.  The second author was supported by NSF grants DMS 1802281 and DMS 2154223.  

\bibliographystyle{habbrv}
\bibliography{bib_all}

\begin{thebibliography}{10}

\bibitem{alberts}
B.~{Alberts}.
\newblock {Explicit analytic continuation of Euler products}, 2024,
  arXiv:2406.18190.

\bibitem{akkm}
S.~Atanasov, N.~Kaplan, B.~Krakoff, and J.~H. Menzel.
\newblock Counting finite index subrings of $\mathbb{Z}^n$.
\newblock {\em Acta Arith.}, 197:221--246, 2021.

\bibitem{brakenhoff}
J.~Brakenhoff.
\newblock {\em Counting problems for number rings}.
\newblock PhD thesis, Leiden University, 2009.

\bibitem{CKK}
G.~Chinta, N.~Kaplan, and S.~Koplewitz.
\newblock The cotype zeta function of $\mathbb{Z}^d$.
\newblock {\em Indag. {M}ath.}, 34(3):643--659, 2023.

\bibitem{cohen_lenstra}
H.~Cohen and H.~W. Lenstra, Jr.
\newblock Heuristics on class groups of number fields.
\newblock In {\em Number theory, {N}oordwijkerhout 1983 ({N}oordwijkerhout,
  1983)}, volume 1068 of {\em Lecture Notes in Math.}, pages 33--62. Springer,
  Berlin, 1984.

\bibitem{DW}
B.~Datskovsky and D.~J. Wright.
\newblock The adelic zeta function associated to the space of binary cubic
  forms. ii. local theory.
\newblock {\em J. Reine Angew. Math.}, 367:27--75, 1986.

\bibitem{delange}
H.~Delange.
\newblock G\'en\'eralisation du th\'eor\`eme de {Ikehara}.
\newblock {\em Annales scientifiques de l'\'Ecole Normale Sup\'erieure}, 3e
  s{\'e}rie, 71(3):213--242, 1954.

\bibitem{dusautoy_grunewald}
M.~du~Sautoy and F.~Grunewald.
\newblock Analytic properties of zeta functions and subgroup growth.
\newblock {\em Annals of Mathematics. Second Series}, 152(3):793--833, 2000.

\bibitem{dusautoy}
M.~{du Sautoy} and L.~Woodward.
\newblock {\em Zeta Functions of Groups and Rings}, volume 1925 of {\em Lecture
  Notes in Mathematics}.
\newblock Springer-Verlag, 2008.

\bibitem{ish_subrings}
K.~Isham.
\newblock Lower bounds for the number of subrings in $\mathbb{Z}^n$.
\newblock {\em J. Num. Theory}, 234:363--390, 2022.

\bibitem{mathematica_work}
K.~Isham and N.~Kaplan.
\newblock Mathematica code \url{https://www.kellyisham.com/research}.

\bibitem{kmt}
N.~Kaplan, J.~Marcinek, and R.~Takloo-Bighash.
\newblock Distribution of orders in number fields.
\newblock {\em Res. Math. Sci.}, 2, 2015.

\bibitem{Liu}
R.~Liu.
\newblock Counting subrings of $\mathbb{Z}^n$ of index $k$.
\newblock {\em J. Combin. Theory Ser. A}, 114:278--299, 2007.

\bibitem{LubotzkySegal}
A.~Lubotzky and D.~Segal.
\newblock {\em Subgroup growth}, volume 212 of {\em Progress in Mathematics}.
\newblock 2003.

\bibitem{Nakagawa}
J.~Nakagawa.
\newblock Orders of a quartic field.
\newblock {\em Mem. Amer. Math. Soc.}, 583, 1996.

\bibitem{narkiewicz1984number}
W.~Narkiewicz.
\newblock {\em Number Theory}.
\newblock World Scientific Publishing Company, 1984.
\newblock Translated from Polish by S. Kanemitsu.

\bibitem{ns}
P.~Q. Nguyen and I.~E. Shparlinski.
\newblock Counting co-cyclic lattices.
\newblock {\em SIAM J. Discrete Math.}, 30(3):1358--1370, 2016.

\bibitem{petro}
V.~M. Petrogradsky.
\newblock Multiple zeta functions and asymptotic structure of free abelian
  groups of finite rank.
\newblock {\em J. Pure Appl. Algebra}, 208(3):1137--1158, 2007.

\bibitem{Voll2010}
C.~Voll.
\newblock Functional equations for zeta functions of groups and rings.
\newblock {\em Ann. of Math. (2)}, 172(2):1181--1218, 2010.

\bibitem{Voll}
C.~Voll.
\newblock A newcomer's guide to zeta functions of groups and rings.
\newblock In {\em Lectures on profinite topics in group theory}, volume~77 of
  {\em London Math. Soc. Stud. Texts}, pages 99--144. Cambridge Univ. Press,
  Cambridge, 2011.

\bibitem{mathematica}
{Wolfram Research{,} Inc.}
\newblock Mathematica, {V}ersion 12.2.
\newblock Champaign, IL, 2020.

\bibitem{WoodICM}
M.~M. Wood.
\newblock Probability theory for random groups arising in number theory.
\newblock In {\em I{CM}---{I}nternational {C}ongress of {M}athematicians.
  {V}ol. 6. {S}ections 12--14}, pages 4476--4508. EMS Press, Berlin, [2023]
  \copyright2023.

\end{thebibliography}

\end{document}